\documentclass{article}
\usepackage{amsmath,amsthm,amssymb,cases,amscd}
\usepackage{indentfirst}
\usepackage{ascmac}
\usepackage{mathtools}
\usepackage{yhmath}
\usepackage{mathrsfs}
\usepackage{textcomp}
\usepackage[toc,page]{appendix}
\usepackage{braket}
\usepackage{enumerate}
\usepackage{multienum}

\makeatletter
    
    \@addtoreset{equation}{section}
\makeatother

\usepackage[top=30truemm,bottom=30truemm,left=20truemm,right=20truemm]{geometry} 
\allowdisplaybreaks
\theoremstyle{plain}

\newtheorem{lem}{Lemma}[section]
\newtheorem{prop}[lem]{Proposition}
\newtheorem{thm}[lem]{Theorem}

\newtheorem{conj}[lem]{Conjecture}
\newtheorem{hyp}[lem]{Hypothesis}

\theoremstyle{definition}

\newtheorem*{ntn}{Notation}
\newtheorem*{ack}{Acknowledgment}

\newcommand{\lra}{\longrightarrow}
\newcommand{\llra}{\longleftrightarrow}
\newcommand{\inj}{\hookrightarrow}
\newcommand{\surj}{\twoheadrightarrow}

\newcommand{\C}{\mathbb{C}}

\newcommand{\Q}{\mathbb{Q}}

\newcommand{\W}{\mathbb{W}}

\newcommand{\Z}{\mathbb{Z}}

\newcommand{\bfk}{\mathbf{k}}
\newcommand{\bfs}{\mathbf{s}}

\newcommand{\rmM}{\mathrm{M}}

\newcommand{\rmb}{\mathrm{b}}
\newcommand{\rmc}{\mathrm{c}}

\newcommand{\myspan}{\operatorname{span}}
\newcommand{\supp}{\operatorname{supp}}
\newcommand{\myim}{\operatorname{Im}}
\newcommand{\myre}{\operatorname{Re}}

\newcommand{\norm}{\operatorname{Norm}}
\newcommand{\cent}{\operatorname{Cent}}
\newcommand{\vol}{\operatorname{vol}}
\newcommand{\inv}{^{-1}}

\newcommand{\calA}{\mathcal{A}}
\newcommand{\calK}{\mathcal{K}}
\newcommand{\calM}{\mathcal{M}}

\newcommand{\calR}{\mathcal{R}}
\newcommand{\calS}{\mathcal{S}}
\newcommand{\calT}{\mathcal{T}}

\newcommand{\scrF}{\mathscr{F}}
\newcommand{\scrH}{\mathscr{H}}
\newcommand{\scrK}{\mathscr{K}}
\newcommand{\scrV}{\mathscr{V}}

\newcommand{\NN}{\mathfrak{N}}

\newcommand{\mmatrix}[4]{\begin{pmatrix} #1 & #2 \\ #3 & #4 \end{pmatrix}}

\newcommand{\WD}{\mathit{WD}}

\newcommand{\an}[1]{\langle #1 \rangle}
\newcommand{\An}[1]{\left\langle #1 \right\rangle}
\newcommand{\gf}{\gamma_F}

\newcommand{\GL}{\operatorname{GL}}
\newcommand{\SL}{\operatorname{SL}}
\newcommand{\Sp}{\operatorname{Sp}}
\newcommand{\Mp}{\operatorname{Mp}}
\newcommand{\ML}{\widetilde{\GL}}
\newcommand{\SO}{\operatorname{SO}}
\newcommand{\Or}{\operatorname{O}}
\newcommand{\calMp}{\mathcal{M}p}

\newcommand{\Hom}{\operatorname{Hom}}
\newcommand{\End}{\operatorname{End}}
\newcommand{\Alt}{\operatorname{Alt}}
\newcommand{\Sym}{\operatorname{Sym}}
\newcommand{\mpsp}[1]{\overline{#1}}

\newcommand{\wl}{\mathrm{W}}
\newcommand{\wlbc}[1]{\mathfrak{S}_{#1} \ltimes (\Z/2\Z)^{#1}}

\newcommand{\Pitemp}{\Pi_{\rm temp}}
\newcommand{\Phitemp}{\Phi_{\rm temp}}
\newcommand{\Irr}{\operatorname{Irr}}
\newcommand{\Ind}{\operatorname{Ind}}
\newcommand{\cv}{\check{v}}

\newcommand{\splw}{{\bf spl}_{\Sp(W)}}
\newcommand{\splv}{{\bf spl}_{\SO(V^+)}}

\newcommand{\wtil}{\widetilde}

\title{Local intertwining relation for metaplectic groups}
\author{Hiroshi Ishimoto}
\date{}

\begin{document}

\maketitle

\begin{abstract}
In an earlier paper of Wee Teck Gan and Gordan Savin, the local Langlands correspondence for metaplectic groups over a nonarchimedean local fields of characteristic zero was established.
In this paper, we formulate and prove a local intertwining relation for metaplectic groups assuming the local intertwining relation for non-quasi-split odd special orthogonal groups.
\end{abstract}

\tableofcontents

\section{Introduction}
In his long-awaited book \cite{art13}, Arthur obtained a classification of irreducible representations of quasi-split symplectic and special orthogonal groups over local fields of characteristic zero (the local Langlands correspondence, which we shall call LLC for short).
Recall the basic form of the correspondence over nonarchimedean local fields of characteristic zero.
Let $F$ be a $p$-adic field, i.e., a finite extension of $\Q_p$, for some prime number $p$.
Let $\Gamma_F$ and $W_F$ be the absolute Galois group and the absolute Weil group of $F$, respectively.
We shall write $\WD_F$ for the Weil-Deligne group $W_F\times\SL_2(\C)$.

Let $G$ be a connected reductive algebraic group defined over $F$.
LLC proposes a classification of irreducible tempered admissible representations of $G(F)$ in terms of tempered admissible $L$-parameters for $G$.
Let $\hat{G}$ be the connected complex Langlands dual group of $G$.
We write $\Pitemp(G)$ for the set of equivalence classes of irreducible tempered admissible representations of $G(F)$, and $\Phitemp(G)$ for the set of equivalence classes of tempered admissible $L$-parameters $\phi : \WD_F \to \hat{G}\rtimes W_F$.
The basic form of LLC is the following:
\begin{conj}
\begin{enumerate}[(1)]
\item
There exists a canonical map
\begin{equation*}
\mathit{LL} : \Pitemp(G) \lra \Phitemp(G)
\end{equation*}
with some important properties.
\item
For each $\phi \in \Phitemp(G)$, the fiber $\Pi_\phi=\Pi_\phi(G)=\mathit{LL}\inv(\phi)$ is a finite set.
It is called a packet.
\end{enumerate}
\end{conj}
There are further expected properties.
We refer the reader to \cite{bor}, \cite{artast}, or \cite{llcnqs} for details.\\

As mentioned above, Arthur \cite{art13} established LLC for quasi-split $\SO_{2n}$, $\SO_{2n+1}$, and $\Sp_{2n}$, which denote even special orthogonal, the odd special orthogonal, and the symplectic groups of rank $n$, respectively.
Moreover, M\oe glin-Renard \cite{mr} gives a classification of irreducible tempered representations of non-quasi-split odd special orthogonal groups over $p$-adic fields, hence LLC of Vogan type.
Recall LLC of Vogan type (\cite[Conjecture 4.15]{vog}) over $p$-adic fields.
Let $G^*$ be a quasi-split connected reductive algebraic group over a $p$-adic field $F$.
LLC of Vogan type treats pure inner twists of $G^*$ at the same time.
For each $\phi \in \Phitemp(G)$, we let $S_\phi=S_\phi(G)$ denote the centralizer $\cent(\myim\phi, \hat{G})$, and let $\pi_0(S_\phi)$ denote its component group.
Then LLC of Vogan type proposes the following:
\begin{conj}
\begin{enumerate}[(1)]
\item
There exists a canonical map
\begin{equation*}
\mathit{LLV} : \bigsqcup_{(\xi,z)} \Pitemp(G) \lra \Phitemp(G^*),
\end{equation*}
as $(\xi,z)$ runs over the isomorphism classes of pure inner twists of $G^*$, i.e., $\xi : G^* \to G$ is an inner twist and $z\in Z^1(\Gamma_F, G^*)$ is a 1-cocycle such that $\xi\inv\circ\sigma\circ\xi\circ\sigma\inv=\operatorname{Ad}(z(\sigma))$ for all $\sigma\in\Gamma_F$.
This map satisfies some important properties.
\item
For each $\phi\in\Phitemp(G^*)$, the fiber $\Pi_\phi=\mathit{LLV}\inv(\phi)$ is a finite set.
\item
For each $\phi\in\Phitemp(G^*)$, there exists a bijective map
\begin{equation*}
\iota : \Pi_\phi \lra \Irr(\pi_0(S_\phi)),
\end{equation*}
where $\Irr(\pi_0(S_\phi))$ denotes the set of equivalence classes of irreducible representations of the finite group $\pi_0(S_\phi)$, and this bijection $\iota$ satisfies the endoscopic character relations and other nice properties.
Moreover, once we fix a Whittaker datum of $G^*$, then the map $\iota$ is uniquely determined.
\end{enumerate}
\end{conj}

In this paper, we consider the metaplectic groups, which are possibly not algebraic groups but whose representation theory is similar to that of algebraic groups.
LLC for metaplectic groups was established by Gan-Savin \cite{gs}, which we now introduce.
The metaplectic group, denoted by $\Mp_{2n}(F)$, is a unique nonlinear two-fold cover of $\Sp_{2n}(F)$ with an exact sequence
\begin{equation*}
1 \lra \{\pm1\} \lra \Mp_{2n}(F) \lra \Sp_{2n}(F) \lra 1.
\end{equation*}
Thus we identify $\Mp_{2n}(F)$ with $\Sp_{2n}(F) \times \{\pm1\}$ as sets.
We say that a representation $\pi$ of $\Mp_{2n}(F)$ is genuine if $\pi((1,-1))$ is not trivial.
Let $\Pitemp(\Mp_{2n})$ be the set of equivalence classes of irreducible genuine tempered admissible representations of $\Mp_{2n}(F)$, and put $\Phitemp(\Mp_{2n})=\Phitemp(\SO_{2n+1})$.
Fix a nontrivial additive character $\psi : F \to \C^1$.
We have LLC for $\Mp_{2n}(F)$ depending on the choice of $\psi$, due to Gan-Savin \cite{gs}:
\begin{thm}
\begin{enumerate}[(1)]
\item
There exists a map
\begin{equation*}
\mathit{LL}_\psi : \Pitemp(\Mp_{2n}) \lra \Phitemp(\Mp_{2n}),
\end{equation*}
with some important properties.
\item
For each $\phi \in \Phitemp(\Mp_{2n})$, the fiber $\Pi_{\phi,\psi}=\mathit{LL}_\psi\inv(\phi)$ is a finite set.
\item
For each $\phi \in \Phitemp(\Mp_{2n})$, there exists a unique bijective map
\begin{equation*}
\iota_\psi : \Pi_{\phi,\psi} \lra \Irr(\pi_0(S_\phi)),
\end{equation*}
which depends on the choice of $\psi$, and this map satisfies some nice properties.\\
\end{enumerate}
\end{thm}

Although in general the map $\mathit{LL}$ may not be bijective, there is a formula that describes how the bijection $\iota$ classifies the elements in a same packet in terms of intertwining operators.
Namely, this formula can distinguish the elements of each packet $\Pi_\phi$ more precisely by means of the eigenvalues of intertwining operators.
We call this formula the local intertwining relation.
This of course is closely related to the endoscopic character relations.
Also, it is related to the global theories such as the trace formula: global intertwining operators appear in the main terms in the trace formula, and local intertwining operators are their local factors.

In \cite{art13}, Arthur proved the local intertwining relation for quasi-split special orthogonal and symplectic groups (\cite[Theorem 2.4.1]{art13}).
Mok \cite{mok} and Kaletha-M\'inguez-Shin-White \cite{kmsw} proved for inner forms of unitary groups.
Our aim in this paper is to formulate and prove a local intertwining relation for $\Mp_{2n}(F)$ under the assumption that the local intertwining relation for the non-quasi-split odd special orthogonal groups holds.\\

Now we explain the local intertwining relation and our result in more detail.
Let $G$ be a classical group defined over $F$, and $P$ a proper parabolic subgroup of $G$ with a Levi subgroup $M$ defined over $F$.
We then have a canonical inclusion $\hat{M} \subset \hat{G}$.
Composing this inclusion and an $L$-parameter for $M$ gives an inclusion $\Phitemp(M) \subset \Phitemp(G)$.
Let $\phi \in \Phitemp(M)$ be an $L$-parameter for $M$, and also regard it as an $L$-parameter for $G$.
Then LLC and LLC of Vogan type conjecture that the packet $\Pi_\phi(G)$ consists of the irreducible constituents of the representations that are parabolically induced from the elements of $\Pi_\phi(M)$.
For simplicity, we shall consider only the Vogan type conjecture.
The local intertwining relation can distinguish these constituents $\pi$ of $\Ind_P^G(\pi_M)$ in terms of the eigenvalues of certain maps for each $\pi_M \in \Pi_\phi(M)$.
The relation asserts that for any $x \in \pi_0(S_\phi)$, one can construct an endomorphism $R_P(x, \pi_M)$ of $\Ind_P^G(\pi_M)$ explicitly such that $R_P(x, \pi_M)$ acts on $\pi$ by a scalar multiplication by $\iota(\pi)(x)$.
In other words, we expect that for any $x \in \pi_0(S_\phi)$, the concretely defined endomorphism
\begin{equation*}
R_P(x, \pi_M) \in \End_G(\Ind_P^G(\pi_M))
\end{equation*}
satisfies
\begin{equation*}
R_P(x, \pi_M)|_\pi = \iota(\pi)(x)
\end{equation*}
for $\pi \subset \Ind_P^G(\pi_M)$.
This endomorphism is called the normalized self-intertwining operator.

In general, not only the proof of the local intertwining relation, but also the definition of the normalized self-intertwining operator is not trivial.
It is because we have to consider some constant factors, such as the $\varepsilon$-factors, Kottwitz sign, and the Langlands constants ($\lambda$-factors), to define the normalizing factors.
In particular $\varepsilon$-factors depend on the representation $\pi_M$, so they are particularly important.
See \cite{artast} or \cite{art13} for detail.

In this paper we treat the case that $G$ is a metaplectic group $\Mp_{2n}$.
We shall define normalized intertwining operators $\calR_P(x, \pi_M)$ for $\Mp_{2n}$ in \S\ref{io} by
\begin{equation}\label{intr}
\calR_P(x, \pi_M)
=\gf(\psi)^{d(x, \pi_M)} \gamma(\tfrac{1}{2}, \phi_x, \psi)\inv \gamma(0, \rho^\vee \circ \phi, \psi) \calM(x, \pi_M),
\end{equation}
where $d(x, \pi_M)$ is a certain nonnegative integer, $\phi_x$ and $\rho^\vee \circ \phi$ are certain $L$-parameters, and $\calM(x, \pi_M)$ is an unnormalized intertwining operator.
Our definition of the normalized intertwining operators resembles that of classical linear algebraic groups, but there are three subtle and important differences.
First, unlike the case of linear algebraic groups, we can find that the Weil index $\gf(\psi)$ appears in the normalizing factors.
This is a constant that depends only on the additive character $\psi$.
Second, the gamma factor $\gamma(\frac{1}{2}, \phi_x, \psi)\inv$ at $\frac{1}{2}$ appears.
Third, the choice of the Haar measure on the unipotent radical of a parabolic subgroup of $\Mp_{2n}(F)$ is slightly different from the case of linear algebraic groups.
These will be treated in \S\ref{measu} and \S\ref{io}.

Then we define the normalized self-intertwining operator $R_P(x, \pi_M)$ in \S\ref{io} by using the normalized intertwining operator \eqref{intr}.
The main theorem (Theorem \ref{lirmp}) is the following:
\begin{thm}
Assume the local intertwining relation for the odd special orthogonal groups (Hypothesis \ref{lirso} below).
Let $\phi \in \Phitemp(M)$ be an $L$-parameter for a Levi subgroup $M$ of a parabolic subgroup $P$ of $\Mp_{2n}(F)$, and $\pi_M \in \Pi_{\phi, \psi}(M)$.
Then for any $x \in \pi_0(S_\phi(\Mp_{2n}))$, the normalized self-intertwining operator
\begin{equation*}
R_P(x, \pi_M) \in \End_{\Mp_{2n}(F)}(\Ind_P^{\Mp_{2n}(F)}(\pi_M))
\end{equation*}
satisfies
\begin{equation*}
R_P(x, \pi_M)|_\pi
=\iota_\psi(\pi)(x)
\end{equation*}
for $\pi \subset \Ind_P^{\Mp_{2n}(F)}(\pi_M)$.
\end{thm}

\begin{ntn}
Let $F$ be a $p$-adic field, and $|-|_F$ the normalized absolute value on $F$.
We shall write $W_F$ and $\WD_F=W_F \times \SL_2(\C)$ for the Weil group and the Weil-Deligne group of $F$, respectively.
We also write $\Gamma_F$ for the Galois group of $F$.
Let $(-,-)_F$ denote the quadratic Hilbert symbol of $F$.
The Hilbert symbol defines a non-degenerate bilinear form on $F^\times/{F^\times}^2$.
Fix a non-trivial additive character $\psi : F \to \C^1=\set{z \in \C | z \overline{z}=1}$.
For any $c \in F$, we shall define an additive character $\psi_c$ of $F$ by
\begin{equation*}
\psi_c(x)=\psi(cx).
\end{equation*}
For a non-degenerate quadratic form $q$ on a finite dimensional vector space over $F$, we write $\gf(\psi \circ q)$ for the unnormalized Weil index of $\psi \circ q$, a character of second degree.
See \cite[Appendix]{rao} for the definition of the Weil index.
Note that, if a quadratic form $q$ is an orthogonal direct sum $q_1\dot{+}q_2$ of two non-degenerate quadratic forms $q_1$ and $q_2$, then
\begin{equation*}
\gf(\psi \circ q)
=\gf(\psi \circ q_1) \gf(\psi \circ q_2).
\end{equation*}
Let us write $\gf(\psi)$ for the unnormalized Weil index of $[x \mapsto \psi(x^2)]$, and $\gf(a, \psi)$ for the normalized Weil index, which is defined by $\gf(a, \psi)=\gf(\psi_a)/\gf(\psi)$ for $a \in F^\times$.
For a totally disconnected locally compact group $G$, let $\Irr(G)$ denote the set of equivalence classes of irreducible smooth admissible representations of $G$.
In this paper, we treat only smooth admissible representations over $\C$, except representations of $\WD_F$.
For simplicity, by representations of $G$, we mean such representations of $G$.
If $G$ is a linear algebraic group (resp. a metaplectic group) over $F$, we shall write $\Pitemp(G)$ for the set of equivalence classes of irreducible tempered representations (resp. irreducible genuine tempered representations) of $G(F)$, and we may write $G=G(F)$ by abuse of notation.
For an algebraic group $H$, we define the component group of $H$ by $\pi_0(H)=H/H^\circ$, where $H^\circ$ is the identity component of $H$.
The connected complex Langlands dual group of a connected reductive linear algebraic group $G$ is denoted by $\hat{G}$.
For any finite dimensional vector space $X$ over $F$, we write $\calS(X)$ for the space of compactly supported locally constant $\C$-valued functions on $X$.
For any representation $\rho$, we write $\rho^\vee$ for its contragredient.
\end{ntn}

\begin{ack}
I would like to thank my supervisor A. Ichino for many advices.
\end{ack}
%%%%%%%%%%%%%%%%%%%%%%%%%%%%%%%%%%%%%%%%%%%%%%%%%%%%%%%%%%%%%%%%%%%%%%%%%%%%%%%%%%%%%%%%%%%%%%
%%%%%%%%%%%%%%%%%%%%%%%%%%%%%%%%%%%%%%%%%%%%%%%%%%%%%%%%%%%%%%%%%%%%%%%%%%%%%%%%%%%%%%%%%%%%%%
\section{Metaplectic and orthogonal groups}
Let us begin with a brief review of the metaplectic and orthogonal groups.
In this section, we fix some notations for the groups of interest in this paper.
\subsection{Symplectic group}\label{symp}
First we introduce some notation on symplectic groups.
Let $(W, \an{-,-}_W)$ be a symplectic vector space of dimension $2n$ over $F$, with the associated symplectic group
\begin{equation*}
\Sp(W)
=\Set{ g \in \GL(W) | \an{ gw, gw' }_W = \an{ w, w' }_W, \quad \forall w, w' \in W }.
\end{equation*}
Choose a symplectic basis $\set{y_1, \dots, y_n, y^*_1, \dots, y^*_n}$ of $W$, and put
\begin{align*}
&Y_k = \myspan_F(y_1, \dots, y_k),&
&Y^*_k = \myspan_F(y^*_1,\dots , y^*_k),&
\end{align*}
for $k=1, \ldots, n$, so that we have a standard complete polarization $W = Y_n \oplus Y^*_n$.
We also let 
\begin{equation*}
W_{n-k}
= \myspan_F(y_{k+1}, \dots,y_n, y^*_{k+1}, \dots y^*_n),
\end{equation*}
so that
\begin{equation*}
W
= Y_k \oplus W_{n-k} \oplus Y^*_k.
\end{equation*}
If $n=0$, then $W=\{0\}$, $\Sp(W)=\{1\}$, and the basis is the empty set.\\

We now describe the parabolic subgroups of $\Sp(W)$ up to conjugacy.
Let $\bfk=(k_1, \ldots, k_m)$ be a sequence of positive integers such that $k_1+\cdots+k_m \leq n$, and put $k_0 = 0$, $n_0 = n-(k_1+\dots+k_m)$.
Consider a flag of isotropic subspaces
\begin{equation*}
Y_{k_1} \subset Y_{k_1+k_2} \subset \dots \subset Y_{k_1+\cdots +k_m}
\end{equation*}
in $Y_n$.
The stabilizer of such a flag is a parabolic subgroup $\mpsp{P_\bfk}$ whose Levi subgroup $\mpsp{M_\bfk}$ is given by
\begin{equation*}
\mpsp{M_\bfk}
\cong \GL_{k_1} \times \cdots \times \GL_{k_m} \times \Sp(W_{n_0}),
\end{equation*}
where $\GL_{k_i}$ is identified with the general linear group of a $k_i$-dimensional space
\begin{equation*}
\myspan_F(y_{k_0+\cdots+k_{i-1}+1}, \dots ,y_{k_0+\cdots+k_{i-1}+k_i}).
\end{equation*}
The reason why we use the overlines for $\mpsp{M_\bfk}$ and $\mpsp{P_\bfk}$ will be clear in the next subsection.
We shall write $N_\bfk$ for the unipotent radical of $\mpsp{P_\bfk}$.
Parabolic subgroups of this form are standard with respect to the splitting $\splw$ defined in \S\ref{weyl}.
Any parabolic subgroup of $\Sp(W)$ is conjugate to a parabolic subgroup of this form.
If $m=1$ and $\bfk=(k)$, we shall write $\mpsp{P_k}$, $\mpsp{M_k}$, and $N_k$ instead of $\mpsp{P_\bfk}$, $\mpsp{M_\bfk}$, and $N_\bfk$, respectively for simplicity.
%%%%%%%%%%%%%%%%%%%%%%%%%%%%%%%%%%%%%%%%%%%%%%%%%%%%%%%%%%%%%%%%%%%%%%%%%%%%%%%%%%%%%%%%%
\subsection{Metaplectic group}\label{mp}
Next we come to metaplectic groups.
If $n=0$, we put $\Mp(W)=\{\pm1\}$.
If $n\geq 1$, then the symplectic group $\Sp(W)$ has a unique nonlinear two-fold central extension $\Mp(W)$, which is called the metaplectic group:
\begin{equation}\label{2mp}
1 \lra \{\pm1\} \lra \Mp(W) \lra \Sp(W) \lra 1.
\end{equation}
As a set, we may write
\begin{equation*}
\Mp(W)
=\Sp(W) \times\{\pm1\}
\end{equation*}
with group law given by
\begin{equation*}
(g,\epsilon) \cdot (g',\epsilon')
= (gg', \epsilon\epsilon' c(g,g')),
\end{equation*}
where $c$ is Ranga Rao's normalized cocycle, which is a 2-cocycle on $\Sp(W)$ valued in $\{\pm1\}$.
See \cite[\S5]{rao} or \cite[\S2]{szp} for detail.
For any subset $A \subset \Sp(W)$, we write $\wtil{A}$ for its preimage under the covering map $\Mp(W) \to \Sp(W)$.
Also, for any subset $B \subset \Mp(W)$, we write $\mpsp{B}$ for its image under the covering map. \\

By the parabolic subgroups of $\Mp(W)$ and their Levi subgroups, we mean the preimages of the parabolic subgroups of $\Sp(W)$ and their Levi subgroups, respectively.
Not only the metaplectic group $\Mp(W)$, but also its parabolic subgroups and Levi subgroups are in general nonlinear.\\

Let us describe the parabolic subgroup $P_\bfk = \wtil{\mpsp{P_\bfk}}$ of $\Mp(W)$, which we shall call a standard parabolic subgroup (with respect to the splitting $\splw$).
The covering \eqref{2mp} splits over the unipotent radical $N_\bfk$ of $\mpsp{P_\bfk}$ by $n \mapsto (n, 1)$ so we may canonically regard $N_\bfk$ as a subgroup of $\Mp(W)$, and one has a Levi decomposition
\begin{equation*}
P_\bfk
= M_\bfk \ltimes N_\bfk,
\end{equation*}
where $M_\bfk = \wtil{\mpsp{M_\bfk}}$ is a Levi subgroup.
The covering $M_\bfk$ over $\mpsp{M_\bfk} \cong \GL_{k_1} \times \GL_{k_2} \times \cdots \times \GL_{k_m} \times \Sp(W_{n_0})$ is given by
\begin{equation*}
M_\bfk
\cong \ML_{k_1} \times_{\mu_2} \cdots \times_{\mu_2} \ML_{k_m} \times_{\mu_2} \Mp(W_{n_0}).
\end{equation*}
Here, the restriction of the covering to $\Sp(W_{n_0})$ is nothing but the metaplectic cover $\Mp(W_{n_0})$ of $\Sp(W_{n_0})$, and the covering over $\GL_{k_i}$ is
\begin{equation*}
\ML_{k_i}
=\GL_{k_i} \times \set{\pm1}
\end{equation*}
with group law
\begin{equation*}
(g,\epsilon) \cdot (g',\epsilon')
= (gg' , \epsilon \epsilon' (\det g, \det g')_F).
\end{equation*}

Let $k$ be a positive integer.
The (genuine) representation theory of $\ML_k$ can be easily related to the representation theory of $\GL_k$.
Indeed, for any irreducible representation $\tau$ of $\GL_k$, we can attach an irreducible genuine representation $\wtil{\tau}$ of $\ML_k$ as in \cite[\S2.4]{gs}, and this attachment $\tau \mapsto \wtil{\tau}$ gives a bijection between $\Irr(\GL_k)$ and $\Irr(\ML_k)$, where $\Irr(\ML_k)$ is the set of equivalence classes of irreducible genuine representations of $\ML_k$.
We stress that this bijection depends on the choice of the additive character $\psi$ because $\wtil{\tau}$ is the twist $\tau \otimes \chi_\psi$ by a genuine character $\chi_\psi$, which is defined by using $\psi$, as in \cite[\S2.4]{gs}.
%%%%%%%%%%%%%%%%%%%%%%%%%%%%%%%%%%%%%%%%%%%%%%%%%%%%%%%%%%%%%%%%%%%%%%%%%%%%%%%%%%%%%%%%%
\subsection{Orthogonal group}\label{orth}
Now we come to the orthogonal groups.
Let $V$ be a $(2n+1)$-dimensional vector space over $F$ equipped with a non-degenerate quadratic form $q=q_V$ of discriminant $1$.
Then we define a symmetric bilinear form $b_q$ associated to $q$ by
\begin{equation*}
b_q(v,v')
= q(v+v') - q(v) -q(v').
\end{equation*}

If $n \geq 1$, up to isomorphism, there are precisely two such quadratic spaces $V$.
One of them, to be denoted by $V^+$, has maximal isotropic subspaces of dimension $n$, whereas the other has maximal isotropic subspaces of dimension $n-1$, to be denoted by $V^-$.
As such, we call the former the split quadratic space and the latter the non-split one.
We shall write
\begin{align*}
\epsilon(V)
=\begin{cases}
+1, & V=V^+; \\
-1, & V=V^-.
\end{cases}
\end{align*}
If $n=0$, we have only one such $V$ up to isomorphism, and put $V^+= V$ and $\epsilon(V)=+1$.\\

Let
\begin{equation*}
\Or(V)
=\Set{ h \in \GL(V) | q(hv) = q(v), \quad \forall v \in V}
\end{equation*}
be the associated orthogonal group.
Then observe that $\Or(V) = \SO(V) \times \{\pm1\}$, where
\begin{equation*}
\SO(V)
=\Or(V) \cap \SL(V)
\end{equation*}
is the special orthogonal group.
The group $\SO(V)$ is split (resp. non-quasi-split) if $V$ is the split (resp. non-split) quadratic space.
If $n \geq 1$, up to isomorphism, there are precisely two pure inner twists of $\SO(V^+)$, namely $\SO(V^+)$ and $\SO(V^-)$.
Note that the Kottwitz sign (\cite{kot}) of $\SO(V)$ is equal to $\epsilon(V)$.\\

Set $r$ to be the dimension of a maximal isotropic subspace of $V$, so that $r=n-(1-\epsilon(V))/2$.
Choose a basis $\set{x_1, \dots, x_n , x_0,   x^*_1, \dots, x^*_n}$ of $V$ such that
\begin{align*}
&b_q(x_i, x_j) = b_q(x^*_i, x^*_j) = 0,&
&b_q(x_i, x^*_j) = \delta_{i,j},& \\
&b_q(x_0, x_i) = b_q(x_0,x^*_i) =0,&
&q(x_0) = 1,&
\end{align*}
for $1\leq i,j \leq r$, and if $r=n-1$,
\begin{equation*}
b_q(x_n,x_i)
=b_q(x_n,x^*_i)
=b_q(x^*_n,x_i)
=b_q(x^*_n, x^*_i)
=0,
\end{equation*}
for any $1\leq i \leq r$.
For each $1\leq k\leq r$, put
\begin{align*}
&X_k = \myspan_F(x_1, \dots, x_k),&
&X^*_k = \myspan_F(x^*_1,\dots , x^*_k),&\\
&V_{n-k} = \myspan_F(x_{k+1}, \dots,x_n, x_0,  x^*_{k+1}, \dots x^*_n),&
\end{align*}
so that
\begin{equation*}
V
=X_k \oplus V_{n-k} \oplus X^*_k.\\
\end{equation*}

We now describe the parabolic subgroups of $\SO(V)$ up to conjugacy.
Let $\bfk=(k_1,\ldots, k_m)$ be a sequence of positive integers such that $k_1+\cdots+k_m \leq r$.
Put $k_0 = 0$ and $n_0 = n - (k_1+ \cdots + k_m)$.
Consider a flag of isotropic subspaces
\begin{equation*}
X_{k_1} \subset X_{k_1+k_2} \subset \dots \subset X_{k_1+\cdots +k_m}
\end{equation*}
in $X_r$.
The stabilizer of such a flag is a parabolic subgroup $Q_\bfk$ whose Levi subgroup $L_\bfk$ is given by
\begin{equation*}
L_\bfk
\cong \GL_{k_1} \times \cdots \times \GL_{k_m} \times \SO(V_{n_0}),
\end{equation*}
where $\GL_{k_i}$ is identified with the general linear group of a $k_i$-dimensional space
\begin{equation*}
\myspan_F(x_{k_0+\cdots+k_{i-1}+1}, \dots ,x_{k_0+\cdots+k_{i-1}+k_i}).
\end{equation*}
We shall write $U_\bfk$ for the unipotent radical of $Q_\bfk$.
Parabolic subgroups of this form are standard with respect to the splitting $\splv$ defined in \S\ref{weyl} if $V=V^+$.
Any parabolic subgroup of $\SO(V)$ is conjugate to a parabolic subgroup of this form.
%%%%%%%%%%%%%%%%%%%%%%%%%%%%%%%%%%%%%%%%%%%%%%%%%%%%%%%%%%%%%%%%%%%%%%%%%%%%%%%%%%%%%%%%%%%%%%
%%%%%%%%%%%%%%%%%%%%%%%%%%%%%%%%%%%%%%%%%%%%%%%%%%%%%%%%%%%%%%%%%%%%%%%%%%%%%%%%%%%%%%%%%%%%%%
\section{Tempered $L$-parameters for $\Mp(W)$ and $\SO(V)$}
In this section, we recall the notion of $L$-parameters for $\Mp(W)$ and $\SO(V)$.
See \cite{ggp} for detail.
\subsection{Symplectic representations of $\WD_F$ and their component groups}
We say that a homomorphism $\phi : \WD_F \to \GL_d(\C)$ is a representation of $\WD_F$ if
\begin{itemize}
\item
$\phi(\mathrm{Frob})$ is semi-simple, where $\mathrm{Frob} \in W_F$ is a geometric Frobenius;
\item
the restriction of $\phi$ to $\SL_2(\C)$ is algebraic;
\item
the restriction of $\phi$ to $W_F$ is smooth.
\end{itemize}
We call $\phi$ tempered if the image of $W_F$ is bounded.
We say that $\phi$ is symplectic if there exists a non-degenerate anti-symmetric bilinear form $B : \C^d \times \C^d \to \C$ such that $B(\phi(w)x, \phi(w)y) = B(x, y)$ for any $x, y \in \C^d$ and $w \in \WD_F$.
In this case, $\phi$ is self-dual.\\

Let $\phi : \WD_F \to \GL_d(\C)$ be a tempered symplectic representation.
By changing bases if necessary, we may assume that $\phi : \WD_F \to \Sp_d(\C)$.
Then, by \cite[\S4]{ggp}, we can write
\begin{equation*}
\phi
=\bigoplus_{i\in I_\phi} \ell_i \phi_i \ \oplus \ (\varphi \oplus \varphi^\vee),
\end{equation*}
where $\ell_i$ are positive integers, $I_\phi$ is an indexing set for mutually inequivalent irreducible symplectic representations $\phi_i$ of $\WD_F$, and $\varphi$ is a representation of $\WD_F$ such that all irreducible summands are non-symplectic.
Let $S_\phi = \cent(\myim \phi, \Sp_d(\C))$ be the centralizer of the image $\myim(\phi)$ in $\Sp_d(\C)$.
Then by \cite[\S4]{ggp}, its component group $\pi_0(S_\phi)$ is canonically identified with a free $\Z/2\Z$-module of rank $\#I_\phi$:
\begin{equation*}
\pi_0(S_\phi)
\cong \bigoplus_{i \in I_\phi} (\Z/2\Z) a_i,
\end{equation*}
where $\{a_i\}$ is a formal basis associated to $\{\phi_i\}$.
In the rest of this paper, we identify $\pi_0(S_\phi)$ with $\oplus_i(\Z/2\Z)a_i$.
We shall write $z_\phi$ for the image of $-1\in S_\phi$ in $\pi_0(S_\phi)$.
%%%%%%%%%%%%%%%%%%%%%%%%%%%%%%%%%%%%%%%%%%%%%%%%%%%%%%%%%%%%%%%%%%%%%%%%%%%%%%%%%%%%%%%%%
\subsection{Tempered $L$-parameters for $\Mp(W)$ and $\SO(V)$}
Let $\Phitemp(\GL_k)$ be the set of equivalence classes of tempered $L$-parameters for $\GL_k$.
Recall that it can be identified with the set of equivalence classes of tempered representations $\phi : \WD_F \to \GL_k(\C)$ of dimension $k$.
Now let $\Phitemp(\Mp_{2n})$ and $\Phitemp(\SO_{2n+1})$ be the set of equivalence classes of tempered $L$-parameters for $\Mp(W)$ and $\SO(V)$, respectively.
Then by \cite[\S11, \S8]{ggp}, we can identify $\Phitemp(\Mp_{2n})$ and $\Phitemp(\SO_{2n+1})$ with the set of equivalence classes of tempered symplectic representations $\phi : \WD_F \to \Sp_{2n}(\C)$ of dimension $2n$.\\

Let $\bfk = (k_1,\ldots,k_m)$ and $n_0$ be as in the previous section.
For $(G,P,M) =(\Mp(W), P_\bfk, M_\bfk)$ or $(\SO(V), Q_\bfk, L_\bfk)$, put
\begin{align*}
&\hat{G} = \Sp_{2n}(\C),&
&\hat{M} = \GL_{k_1}(\C) \times \cdots \times \GL_{k_m}(\C) \times \Sp_{2n_0}(\C),&
\end{align*}
with a standard embedding $\hat{M} \inj \hat{G}$ as a Levi subgroup of a standard parabolic subgroup $\hat{P}$ of $\hat{G}$.
Let $\phi$ be a tempered $L$-parameter for $G$ with the image $\myim(\phi)$ in $\hat{M}$.
This is of the form
\begin{equation}\label{phim}
\phi
= \phi_1 \oplus \cdots \oplus \phi_m \oplus \phi_0 \oplus \phi_m^\vee \oplus \cdots \oplus \phi_1^\vee
\end{equation}
where $\phi_i \in \Phitemp(\GL_{k_i})$ for $i=1,\ldots,m$, and $\phi_0 \in \Phitemp(\Mp_{2n_0})=\Phitemp(\SO_{2n_0+1})$.
Let $A_{\hat{M}}$ be the maximal central torus of $\hat{M}$.
Put
\begin{align*}
\NN_\phi(M, G)
&= \norm(A_{\hat{M}}, S_\phi) / \cent(A_{\hat{M}}, S_\phi^\circ),\\
\wl_\phi(M, G)
&= \norm(A_{\hat{M}}, S_\phi) / \cent(A_{\hat{M}}, S_\phi),\\
S_\phi^\natural(M, G)
&= \norm(A_{\hat{M}}, S_\phi) / \norm(A_{\hat{M}}, S_\phi^\circ).
\end{align*}
We have a natural surjection
\begin{equation}\label{x}
\NN_\phi(M,G) \surj S_\phi^\natural(M,G),
\end{equation}
natural inclusions
\begin{align*}
&\wl_\phi(M,G) \subset \wl(\hat{M}, \hat{G}),&
&S_\phi^\natural(M, G) \subset \pi_0(S_\phi),&
\end{align*}
and a natural short exact sequence
\begin{equation*}
1 \lra \pi_0(S_{\phi_0}) \lra \NN_\phi(M, G) \lra \wl_\phi(M, G) \lra 1.
\end{equation*}
By applying \cite[p.104]{art13} or \cite[p.103, after (2.4.1)]{kmsw} to $\SO(V)$, the injection $\pi_0(S_{\phi_0}) \to \NN_\phi(M, G)$ admits a canonical splitting
\begin{equation*}
\NN_\phi(M, G) = \pi_0(S_{\phi_0}) \times \wl_\phi(M, G).
\end{equation*}
%%%%%%%%%%%%%%%%%%%%%%%%%%%%%%%%%%%%%%%%%%%%%%%%%%%%%%%%%%%%%%%%%%%%%%%%%%%%%%%%%%%%%%%%%%%%%%
%%%%%%%%%%%%%%%%%%%%%%%%%%%%%%%%%%%%%%%%%%%%%%%%%%%%%%%%%%%%%%%%%%%%%%%%%%%%%%%%%%%%%%%%%%%%%%
\section{Local Langlands correspondence for $\Mp(W)$ and the main theorem}
In this section, we summarize some properties of the local Langlands correspondence (LLC) for metaplectic groups, and state the main theorem (Theorem \ref{lirmp}).
The correspondence is defined by combining the local Shimura correspondence with LLC for odd special orthogonal groups, which we shall summarize in \S\ref{Llclirso} and\S\ref{Lsc} below.

The local Langlands correspondence for metaplectic groups was established by Gan-Savin.
(\cite[Corollary 1.2, Theorem 1.3]{gs}, and \cite{han} for the last assertion):
\begin{thm}\label{llcmp}
\begin{enumerate}[(1)]
\item
There exists a surjection (depending on $\psi$)
\begin{equation*}
\mathit{LL}_\psi : \Pitemp(\Mp(W)) \lra \Phitemp(\Mp_{2n}),
\end{equation*}
with finite fibers $\Pi_{\phi, \psi} = \Pi_{\phi, \psi}(\Mp(W)) = \mathit{LL}_\psi\inv(\phi)$.
\item
For each $\phi \in \Phitemp(\Mp_{2n})$, there exists a unique bijection (depending on $\psi$)
\begin{equation*}
\iota_\psi : \Pi_{\phi, \psi} \lra \Irr(\pi_0(S_\phi)).
\end{equation*}
\item
Let $\bfk = (k_1, \ldots, k_m)$, $n_0=n-(k_1+\cdots+k_m) \geq 0$, and let $\phi \in \Phitemp(\Mp_{2n})$ be of the form \eqref{phim}.
Then we have
\begin{align*}
\Pi_{\phi, \psi}
=\Set{
\pi
|
\pi \subset \Ind_{P_\bfk}^{\Mp(W)}(\wtil{\tau_1} \otimes \cdots \otimes \wtil{\tau_m} \otimes \pi_0), \text{irreducible constituent},
\quad \pi_0 \in \Pi_{\phi_0, \psi} 
},
\end{align*}
where $\tau_i$ is the representation of $\GL_{k_i}$ which corresponds to $\phi_i$, $i=1,\ldots, m$.
Moreover for any $\pi_0 \in \Pi_{\phi_0, \psi}$, we have
\begin{align*}
\Ind_{P_\bfk}^{\Mp(W)}(\wtil{\tau_1} \otimes \cdots \otimes \wtil{\tau_m} \otimes \pi_0)
=\bigoplus_{\substack{\pi \in \Pi_{\phi, \psi} \\ \iota_\psi(\pi)|_{\pi_0(S_{\phi_0})} = \iota_\psi(\pi_0)}}  \pi.
\end{align*}
\end{enumerate}
\end{thm}

In the setting of Theorem \ref{llcmp} (3), for $w \in \wl_\phi(M_{\bfk},\Mp(W))$, let
\begin{align*}
R_{P_\bfk}(w, \wtil{\tau_1} \otimes \cdots \otimes \wtil{\tau_m} \otimes \pi_0)
\in
\End_{\Mp(W)}(\Ind^{\Mp(W)}_{P_\bfk}(\wtil{\tau_1} \otimes \cdots \otimes \wtil{\tau_m} \otimes \pi_0))
\end{align*}
be the normalized self-intertwining operator defined in \S\ref{io} below.
Then, we can state the main theorem:
\begin{thm}\label{lirmp}
Assume the local intertwining relation for the odd special orthogonal groups (Hypothesis \ref{lirso} below).
Let $x_w \in S_\phi^\natural(M_\bfk,\Mp(W))$ be the image of $w \in \wl_\phi(M_\bfk,\Mp(W))$ under the natural surjection \eqref{x}.
Then, the restriction of $R_{P_\bfk}(w, \wtil{\tau_1} \otimes \cdots \otimes \wtil{\tau_m} \otimes \pi_0)$ to $\pi \subset \Ind_{P_\bfk}^{\Mp(W)}(\wtil{\tau_1} \otimes \cdots \otimes \wtil{\tau_m} \otimes \pi_0)$ is the scalar multiplication by $\iota_\psi(\pi)(x_w)$.
\end{thm}

We will reduce the main theorem to Proposition \ref{reduction} in \S\ref{red}, and complete a proof of the proposition in \S\ref{pf}.
%%%%%%%%%%%%%%%%%%%%%%%%%%%%%%%%%%%%%%%%%%%%%%%%%%%%%%%%%%%%%%%%%%%%%%%%%%%%%%%%%%%%%%%%%%%%%%
%%%%%%%%%%%%%%%%%%%%%%%%%%%%%%%%%%%%%%%%%%%%%%%%%%%%%%%%%%%%%%%%%%%%%%%%%%%%%%%%%%%%%%%%%%%%%%
\section{Local Langlands correspondence and the local intertwining relation for $\SO(V)$}\label{Llclirso}
The local Langlands correspondence for odd special orthogonal groups was established by Arthur \cite{art13} and M\oe glin-Renard \cite{mr}.
In this section, we summarize some properties of the correspondence and the local intertwining relation.

Arthur \cite{art13} and M\oe glin-Renard \cite{mr} studied representations of $\SO(V^+)$ and $\SO(V^-)$, respectively.
Their results imply LLC of Vogan type for $\SO(V)$:
\begin{thm}\label{llcso}
\begin{enumerate}[(1)]
\item
There exists a surjection
\begin{equation*}
\mathit{LLV} : \Pitemp(\SO(V^+)) \bigsqcup \Pitemp(\SO(V^-)) \lra \Phitemp(\SO_{2n+1}),
\end{equation*}
with finite fibers $\Pi_\phi = \Pi_\phi(\SO(V^+)) \sqcup \Pi_\phi(\SO(V^-)) = \mathit{LLV}\inv(\phi)$.
\item
For each $\phi \in \Phitemp(\SO_{2n+1})$, there exists a unique bijective map
\begin{equation*}
\iota : \Pi_\phi \longrightarrow \Irr(\pi_0(S_\phi))
\end{equation*}
such that
\begin{equation*}
\Pi_\phi(\SO(V^\pm))
=\Set{ \sigma \in \Pi_\phi |  \iota(\sigma)(z_\phi) = \pm 1 }.
\end{equation*}
\item
Let $V=V^+$ or $V^-$.
Let $\bfk = (k_1, \ldots, k_m)$ be a sequence of positive integers such that $k_1+\cdots+k_m \leq r$, and put $n_0=n-(k_1+\cdots+k_m)$.
Let $\phi \in \Phitemp(\SO_{2n+1})$ be of the form \eqref{phim}.
Then we have
\begin{equation*}
\Pi_\phi
=\Set{
\sigma
|
\sigma \subset \Ind_{Q_\bfk}^{\SO(V)}(\tau_1 \otimes \cdots \otimes \tau_m \otimes \sigma_0), \text{irreducible constituent},
\quad \sigma_0 \in \Pi_{\phi_0}
},
\end{equation*}
where $\tau_i$ is the representation of $\GL_{k_i}$ which corresponds to $\phi_i$, $i=1,\ldots, m$.
Moreover for any $\sigma_0 \in \Pi_{\phi_0}$, we have
\begin{align*}
\Ind_{Q_\bfk}^{\SO(V)}(\tau_1 \otimes \cdots \otimes \tau_m \otimes \sigma_0)
=\bigoplus_{\substack{\sigma \in \Pi_\phi \\ \iota(\sigma)|_{\pi_0(S_{\phi_0})} = \iota(\sigma_0)}} \sigma.
\end{align*}
\end{enumerate}
\end{thm}

In the setting of Theorem \ref{llcso} (3), for $w \in \wl_\phi(L_\bfk, \SO(V))$, let
\begin{equation*}
R_{Q_\bfk}(w, \tau_1 \otimes \cdots \otimes \tau_m \otimes \sigma_0)
\in \End_{\SO(V)}(\Ind_{Q_\bfk}^{\SO(V)}(\tau_1 \otimes \cdots \otimes \tau_m \otimes \sigma_0))
\end{equation*}
be the normalized self-intertwining operator defined in \S\ref{io} below.
The next hypothesis is the local intertwining relation for $\SO(V)$, and it has already been proven in the case $V=V^+$ by Arthur \cite[\S2.4]{art13}.
\begin{hyp}\label{lirso}
Let $x_w \in S_\phi^\natural(L_\bfk, \SO(V))$ be the image of $w \in \wl_\phi(L_\bfk, \SO(V))$ under the natural surjection \eqref{x}.
Then, the restriction of $R_{Q_\bfk}(w, \tau_1 \otimes \cdots \otimes \tau_m \otimes \sigma_0)$ to $\sigma \subset \Ind_{Q_\bfk}^{\SO(V)}(\tau_1 \otimes \cdots \otimes \tau_m \otimes \sigma_0)$ is the scalar multiplication by $\iota(\sigma)(x_w)$.
\end{hyp}
%%%%%%%%%%%%%%%%%%%%%%%%%%%%%%%%%%%%%%%%%%%%%%%%%%%%%%%%%%%%%%%%%%%%%%%%%%%%%%%%%%%%%%%%%%%%%%
%%%%%%%%%%%%%%%%%%%%%%%%%%%%%%%%%%%%%%%%%%%%%%%%%%%%%%%%%%%%%%%%%%%%%%%%%%%%%%%%%%%%%%%%%%%%%%
\section{Local Shimura correspondence}\label{Lsc}
Gan-Savin \cite{gs} showed the local Shimura correspondence, which is the natural bijection between the set of isomorphism classes of irreducible genuine representations of $\Mp(W)$ and the set of isomorphism classes of irreducible representations of $\SO(V^+)$ and $\SO(V^-)$.
This is given by the local theta correspondence, and we can construct LLC for $\Mp(W)$ (Theorem \ref{llcmp}), by combining the local Shimura correspondence with LLC for $\SO(V)$ of Vogan type (Theorem \ref{llcso}).
In this section, we shall review their results.
First we recall the Weil representation for $\Mp(W) \times \Or(V)$ and the notion of the local theta correspondence.
\subsection{Weil representation}
The group $\Mp(W) \times \Or(V)$ has a natural representation $\omega_{V, W, \psi}$ depending on $\psi$, given as follows.
The tensor product $\W=V \otimes _F W$ has a natural symplectic form $\an{-,-}$ defined by
\begin{equation*}
\an{ v\otimes w, v'\otimes w' } = b_q(v, v') \cdot \an{ w, w' }_W.
\end{equation*}
Then there is a natural map
\begin{equation}\label{natsp}
\Sp(W) \times \Or(V) \lra \Sp(\W).
\end{equation}
One has the metaplectic $\C^1$-cover $\calMp(\W)$ of $\Sp(\W)$, and the additive character $\psi$ determines the Weil representation $\omega_\psi$ of $\calMp(\W)$.
Kudla \cite{spl} gives a splitting of the metaplectic cover over $\Mp(W) \times \Or(V)$, hence there exists a commutative diagram
\begin{align*}
\begin{CD}
\Mp(W) \times \Or(V)  @>>> \calMp(\W)  \\
    @VVV                                @VVV  \\
\Sp(W) \times \Or(V)  @>>> \Sp(\W),
\end{CD}
\end{align*}
where the right vertical map is given by the metaplectic $\C^1$-covering map, the left vertical map is given by the two-fold cover \eqref{2mp}, and the lower horizontal map is \eqref{natsp}.
Thus, we have a Weil representation $\omega_{V, W, \psi}$ of $\Mp(W) \times \Or(V)$.
We will later in \S\ref{mixmodel} give some realizations of the Weil representation $\omega_{V, W, \psi}$ to show the main theorem.
Here, a splitting over $\Mp(W) \times \Or(V)$ is not unique, and we choose one following \cite{spl}.
%%%%%%%%%%%%%%%%%%%%%%%%%%%%%%%%%%%%%%%%%%%%%%%%%%%%%%%%%%%%%%%%%%%%%%%%%%%%%%%%%%%%%%%%%
\subsection{Local theta correspondence}
In this subsection, we summarize the result of Gan-Savin \cite{gs}.
First note that the theorems in \cite{gs} had been verified only for odd residual characteristic since the Howe duality for even residue characteristic was conjecture then.
However, the Howe duality for even residue characteristic was verified by Gan-Takeda \cite{gt}, so now we have the results of \cite{gs} for arbitrary residue characteristic.

Given an irreducible representation $\sigma$ of $\Or(V)$, the maximal $\sigma$-isotypic quotient of $\omega_{V,W,\psi}$ is of the form
\begin{equation*}
\sigma \boxtimes \Theta_{V,W,\psi}(\sigma)
\end{equation*} 
for some representation $\Theta_{V,W,\psi}(\sigma)$ of $\Mp(W)$ (called the big theta lift of $\sigma$).
Then $\Theta_{V,W,\psi}(\sigma)$ is either zero or has finite length.
The maximal semisimple quotient of $\Theta_{V,W,\psi}(\sigma)$ is denoted by $\theta_{V,W,\psi}(\sigma)$ (called the small theta lift of $\sigma$).

Similarly, if $\pi$ is an irreducible genuine representation of $\Mp(W)$, then one has its big theta lift $\Theta_{W,V,\psi}(\pi)$ and its small theta lift $\theta_{W,V,\psi}(\pi)$, which are representations of $\Or(V)$.\\

By the Howe duality, each small theta lift is irreducible or zero (\cite{wal}, \cite{gt}).
Gan-Savin \cite[\S6]{gs} showed that
\begin{enumerate}
\item
for $\pi \in \Irr(\Mp(W))$, exactly one of $\theta_{W,V^+,\psi}(\pi)$ or $\theta_{W,V^-,\psi}(\pi)$ is nonzero;
\item
given $\sigma \in \Irr(\SO(V))$, with the extensions $\sigma^+$ and $\sigma^-$ to $\Or(V)$, exactly one of $\Theta_{V,W,\psi}(\sigma^+)$ or $\Theta_{V,W,\psi}(\sigma^-)$ is nonzero,
\end{enumerate}
where $\Irr(\Mp(W))$ is the set of equivalence classes of irreducible genuine representations of $\Mp(W)$, and $\sigma^\pm$ denote the extensions such that $-1 \in \Or(V)$ acts as $\pm1$, respectively.
Then they derived the following theorems (\cite[Theorem 1.1, Theorem 1.3]{gs}):
\begin{thm}
There is a bijection
\begin{equation*}
\Theta_\psi : \Irr(\Mp(W)) \llra \Irr(\SO(V^+)) \sqcup \Irr(\SO(V^-)),
\end{equation*}
given by the theta correspondence with respect to $\psi$.
\end{thm}
\begin{thm}\label{sw}
Suppose that $\sigma \in \Irr(\SO(V))$ and $\pi \in \Irr(\Mp(W))$ correspond under $\Theta_\psi$.
Then we have the following.
\begin{enumerate}[(1)]
\item
$\sigma$ is a discrete series representation if and only if $\pi$ is a discrete series representation.
\item
$\sigma$ is tempered if and only if $\pi$ is tempered.
Moreover, suppose that
\begin{equation*}
\sigma \subset \Ind_{Q_\bfk}^{\SO(V)}(\tau_1\otimes \cdots \otimes \tau_m \otimes \sigma_0),
\end{equation*}
where $\bfk =(k_1,\ldots,k_m)$ is a sequence such that $k_1+\cdots+k_m \leq r$, $\tau_i$'s are tempered representations of $\GL_{k_i}$, $\sigma_0$ is a tempered representation of $\SO(V_{n_0})$, and $n_0 = n-(k_1+\cdots+k_m)$.
Then
\begin{equation*}
\pi \subset \Ind_{P_\bfk}^{\Mp(W)}(\wtil{\tau_1}\otimes \cdots \otimes \wtil{\tau_m} \otimes \pi_0),
\end{equation*}
where $\pi_0 = \Theta_\psi(\sigma_0)$.
In particular, $\Theta_\psi$ gives a bijection between the (isomorphism classes of) irreducible constituents of $\Ind_{Q_\bfk}^{\SO(V)}(\tau_1\otimes \cdots \otimes \tau_m \otimes \sigma_0)$ and those of $\Ind_{P_\bfk}^{\Mp(W)}(\wtil{\tau_1}\otimes \cdots \otimes \wtil{\tau_m} \otimes \pi_0)$.
\item
If $\sigma$ is an irreducible representation of $\SO(V)$ and $\rho$ is an irreducible representation of $\GL_l$, then one has a Plancherel measure $\mu(s,\sigma \times \rho, \psi)$ associated to the parabolically induced representation $\Ind_{Q_l}^{\SO(V)}(\rho_s \otimes \sigma)$, where $\rho_s=\rho \otimes |\det|_F^s$.
If $\pi = \Theta_\psi(\sigma)$, then one has
\begin{equation*}
\mu(s, \sigma \times \rho, \psi)
=\mu(s, \pi \times \rho, \psi).
\end{equation*}
\end{enumerate}
\end{thm}

See \cite[Appendix B]{gifd} and \cite[Appendix A.7]{gigp} for detail of Plancherel measures.

By combining Theorem \ref{sw} with Theorem \ref{llcso} and \cite{han}, we obtain Theorem \ref{llcmp}.
%%%%%%%%%%%%%%%%%%%%%%%%%%%%%%%%%%%%%%%%%%%%%%%%%%%%%%%%%%%%%%%%%%%%%%%%%%%%%%%%%%%%%%%%%%%%%%
%%%%%%%%%%%%%%%%%%%%%%%%%%%%%%%%%%%%%%%%%%%%%%%%%%%%%%%%%%%%%%%%%%%%%%%%%%%%%%%%%%%%%%%%%%%%%%
\section{Intertwining operators}\label{Io}
In this section, we shall define the normalized self-intertwining operators of $\SO(V)$ (following \cite[\S2.3]{art13}) and those of $\Mp(W)$, which are used in Theorem \ref{lirmp} and Hypothesis \ref{lirso} above.
The definition of the normalized self-intertwining operators is very subtle because one has to choose the
following data appropriately:
\begin{itemize}
\item
representatives of a Weyl group element $w$;
\item
Haar measures on the unipotent radicals to define the unnormalized intertwining operators;
\item
normalizing factors $r_P(w, \pi_{M, \bfs})$ and $r_Q(w, \sigma_{L, \bfs})$;
\item
an intertwining isomorphism $\mathcal{A}_w$.
\end{itemize}

Let $\bfk = (k_1, \ldots, k_m)$ be a sequence of positive integers such that $k_1+ \cdots + k_m \leq r$.
Put $k_0=0$, $k=k_1+\cdots +k_m$, $n_0 = n-k$, and
\begin{align*}
 P&=P_{\bfk},&  M&=M_{\bfk},&  N&=N_{\bfk},& &\quad&
 Q&=Q_{\bfk},&  L&=L_{\bfk},&  U&=U_{\bfk}.&
\end{align*}
\subsection{Representatives of a Weyl group element}\label{weyl}
Let $w \in \wl(\hat{M}, \Sp_{2n}(\C))$ be a Weyl group element.
We shall identify $w$ with elements in $\wl(M, \Mp(W))$, $\wl(L, \SO(V))$ in a standard way, and take representatives $\wtil{w}_P \in \Mp(W)$ and $\wtil{w}_Q \in \SO(V)$ following Langlands-Shelstad \cite{ls} and Gan-Li \cite{ganli}.
In this subsection, we review the procedure.
See \cite[\S2.1]{ls}, \cite[\S2.3]{art13}, or \cite[Definition 4.1]{ganli} for detail.\\

First, we realize the relative Weyl group $\wl(\hat{M}, \Sp_{2n}(\C))$ in $\wlbc{m}$, and identify it with relative Weyl groups $\wl(\mpsp{M}, \Sp(W))$ and $\wl(L, \SO(V))$.
We can do this in a canonical way, because the Levi subgroups $\hat{M}$, $\mpsp{M}$, and $L$ have a form
\begin{equation*}
\GL_{k_1} \times \cdots \times \GL_{k_m} \times G_-
\end{equation*}
over $\C$ or $F$, where $G_-$ is a semi-simple algebraic group.

Next, we take standard splittings $\splw$ of $\Sp(W)$ and $\splv$ of $\SO(V^+)$ by
\begin{align*}
&\splw = (B_W, T_W , \{X_{\alpha_i}\}_{i=1,\dots,n}),&
&\splv = (B_V, T_V, \{X_{\beta_i}\}_{i=1,\dots,n}),&
\end{align*}
where
\begin{itemize}
\item
$B_W$ and $B_V$ are respectively the Borel subgroups stabilizing the $F$-flags
\begin{align*}
&Fy_1  \subset Fy_1 + Fy_2 \subset \cdots \subset Fy_1 + \cdots +Fy_n,&
&Fx_1  \subset Fx_1 + Fx_2 \subset \cdots \subset Fx_1 + \cdots +Fx_n;&
\end{align*}
\item
$T_W$ and $T_V$ are respectively their maximal tori which are diagonalized by the bases
\begin{align*}
&\set{y_1,\dots,y_n,y_1^*,\dots,y_n^*},&
&\set{x_1,\dots,x_n,x_0,x_1^*,\dots,x_n^*};&
\end{align*}
\item
$X_{\alpha_i}$ and $X_{\beta_i}$ are simple root vectors given as follows:
\begin{align*}
X_{\alpha_i} : y_j &\mapsto \begin{cases}  y_i & j=i+1, \\
                                                        0   &  \text{otherwise},
                                    \end{cases}&
                y^*_j&\mapsto \begin{cases} -y^*_{i+1} & j=i, \\
                                                          0           &\text{otherwise},
                                       \end{cases}& \text{for }1\leq i\leq n-1,\\
X_{\alpha_n} : y_j &\mapsto 0, &
                y^*_j&\mapsto \begin{cases} y_n & j=n, \\
                                                      0     &\text{otherwise},
                                     \end{cases}& \\
X_{\beta_i} : x_j &\mapsto \begin{cases}  x_i & j=i+1, \\
                                                      0   &\text{otherwise},
                                   \end{cases}&
              x^*_j&\mapsto \begin{cases} -x^*_{i+1} & j=i, \\
                                                         0         &\text{otherwise},
                                     \end{cases}& \text{for }1\leq i\leq n-1,\\
X_{\beta_n} : x_j &\mapsto   \begin{cases} 2x_n & j=0, \\
                                                       0     &\text{otherwise},
                                      \end{cases}&
                x^*_j&\mapsto \begin{cases} -x_0 & j=n, \\
                                                         0           &\text{otherwise}.
                                       \end{cases}&
\end{align*}
\end{itemize}
Then $M$ and $P$ (resp. $L$ and $Q$) are standard, in the sense that they contain $\wtil{T_W}$ and $\wtil{B_W}$ (resp. $T_V$ and $B_V$) respectively.
Let $\Phi(T_W,\Sp(W))$ and $\Delta(B_W)=\{\alpha_1,\ldots,\alpha_n\}$ denote the set of roots and the set of simple positive roots with the indices relative to the basis.
Then we can see that the simple root vectors $X_{\alpha_i}$ are indeed corresponding to $\alpha_i$.
Let $X_{-\alpha_i}$ be the root vector for $-\alpha_i$ such that the Lie bracket $[X_{\alpha_i}, X_{-\alpha_i}]$ is the coroot for $\alpha_i$.
Let us take $\Phi(T_V,\SO(V^+))$, $\Delta(B_V)=\{\beta_1,\ldots,\beta_n \}$, and $X_{-\beta_i}$ similarly.\\

Now let us take representatives $\wtil{w}_P$ and $\wtil{w}_Q$ of $w$.
First assume that $V=V^+$.
Let $w_T$ and $w_T'$ denote the representatives of $w$ in the Weyl groups $\wl(T_W,\Sp(W))$ and $\wl(T_V, \SO(V^+))$ that stabilize the simple positive roots inside $\mpsp{M}$ and $L$ respectively.
We shall write $w_\lambda$ for the reflection corresponding to a root $\lambda$.
We then have Langlands-Shelstad's representatives
\begin{align*}
\wtil{w}_P &= \wtil{w}_{\alpha_{(1)}} \cdots \wtil{w}_{\alpha_{(\ell)}},&
\wtil{w}_Q &= \wtil{w}_{\beta_{(1)}} \cdots \wtil{w}_{\beta_{(\ell)}},
\end{align*}
where $w_T = w_{\alpha_{(1)}}\cdots w_{\alpha_{(\ell)}}$ and $w_T'=w_{\beta_{(1)}}\cdots w_{\beta_{(\ell)}}$ are reduced decompositions of $w_T$ and $w_T'$ in $\wl(T_W,\Sp(W))$ and $\wl(T_V, \SO(V^+))$ respectively, and
\begin{align*}
&\wtil{w}_{\alpha} = \left( \exp(X_\alpha)\exp(-X_{-\alpha})\exp(X_\alpha) , 1\right),&
&\wtil{w}_{\beta} = \exp(X_\beta)\exp(-X_{-\beta})\exp(X_\beta),&
\end{align*}
for any $\alpha \in \Delta(B_W)$, $\beta \in \Delta(B_V)$.

In case $V=V^-$, the representative $\wtil{w}_Q \in \SO(V^-)$ is defined to be the corresponding element via the canonical pure inner twist $(\xi, z) : \SO(V^+) \to \SO(V^-)$.
The following lemma is obvious but important.
\begin{lem}
Let $w=w_1 \cdots w_l$ be a reduced decomposition of $w$ as an element of the relative Weyl group $\wl(\hat{M},\Sp_{2n}(\C))$.
Then we have
\begin{align*}
\wtil{w}_P&=\wtil{w_1}_P \cdots \wtil{w_l}_P&
 &\text{and}&
\wtil{w}_Q&=\wtil{w_1}_Q \cdots \wtil{w_l}_Q.
\end{align*}
\end{lem}

The representatives can be given more explicitly in the case $m=1$. See Proposition \ref{LS} below.
%%%%%%%%%%%%%%%%%%%%%%%%%%%%%%%%%%%%%%%%%%%%%%%%%%%%%%%%%%%%%%%%%%%%%%%%%%%%%%%%%%%%%%%%%
\subsection{Haar measures on the unipotent radicals}\label{measu}
In this subsection, we shall choose Haar measures on $N$ and $U$.
We first define Haar measures $du$ on $U$ for $V=V^+$, and $d'n$ on $N$ with respect to the splittings $\splv$ and $\splw$, following \cite[\S2.3]{art13} or \cite[\S2.2]{kmsw}.
Here, $d'n$ is a Haar measure when we consider $N$ as the unipotent radical of a parabolic subgroup $\mpsp{P}$ of $\Sp(W)$.
On the unipotent radical $N$ of the parabolic subgroup $P$ of $\Mp(W)$, we take a Haar measure $dn = |2|_F^{-k/2} d'n$.

Since the splittings are given explicitly, one can describe these measures explicitly.
We will give an explicit definition of the measures in the case $m=1$, i.e., $\bfk=(k)$, in \S\ref{measm} below.
The measures for $m=1$ give us the following descriptions of $du$ and $dn$.

For $1 \leq i\leq m$, put $n_i = n-(k_0+\cdots+k_{i-1})$.
As in \S\ref{symp} and \ref{orth}, let $N^{(i)}$ and $U^{(i)}$ be the unipotent radicals of the maximal parabolic subgroups of $\Sp(W_{n_i})$ and $\SO(V_{n_i})$ stabilizing
\begin{align*}
 &\myspan_F(y_{k_0+\cdots+k_{i-1}+1}, \ldots, y_{k_0+\cdots+k_i}),&
 &\myspan_F(x_{k_0+\cdots+k_{i-1}+1}, \ldots, x_{k_0+\cdots+k_i}),&
\end{align*}
respectively.
If we take Haar measures on each $N^{(i)}$ and $U^{(i)}$ as we will do on $N_k$ and $U_k$ in \S \ref{measm} below, then the Haar measures $dn$ on $N$ and $du$ on $U$ are the measures defined via the following homeomorphisms
\begin{align*}
 N^{(1)} \times \cdots \times N^{(m)} &\lra N,& (n_1, \ldots, n_m)&\mapsto n_1 \cdots n_m,\\
 U^{(1)} \times \cdots \times U^{(m)} &\lra U,& (u_1, \ldots, u_m)&\mapsto u_1 \cdots u_m.
\end{align*}
In case $V=V^-$, we shall define the Haar measure $du$ on $U$ by using the above homeomorphism.
%%%%%%%%%%%%%%%%%%%%%%%%%%%%%%%%%%%%%%%%%%%%%%%%%%%%%%%%%%%%%%%%%%%%%%%%%%%%%%%%%%%%%%%%%
\subsection{Intertwining operators}\label{io}
Now we shall define intertwining operators.
Let $\tau_i$ be irreducible tempered representations of $\GL_{k_i}$ on a vector space $\scrV_{\tau_i}$, for $i=1, \ldots, m$.
For any $\bfs=(s_1, \ldots, s_m) \in \C^m$, we realize the representation $\tau_{i, s_i}=\tau_i \otimes |\det|_F^{s_i}$ on $\scrV_{\tau_i}$ by setting $\tau_{i, s_i}(a)v=|\det a|_F^{s_i} \tau_i(a)v$ for $v \in \scrV_{\tau_i}$ and $a \in \GL_{k_i}$.
Let $\pi_0$ be an irreducible genuine tempered representation of $\Mp(W_{n_0})$ on $\scrV_{\pi_0}$, and $\sigma_0$ an irreducible tempered representation of $\SO(V_{n_0})$ on $\scrV_{\sigma_0}$, such that $\pi_0$ and $\sigma_0$ correspond under the bijection $\Theta_\psi$.
Put $\pi_{M,\bfs} = \wtil{\tau}_{1,s_1}\otimes\cdots\otimes\wtil{\tau}_{m,s_m} \otimes \pi_0$, and $\sigma_{L,\bfs} = \tau_{1,s_1}\otimes\cdots\otimes\tau_{m,s_m} \otimes \sigma_0$.
Especially, we shall write $\pi_M=\pi_{M,0}$ and $\sigma_L=\sigma_{L,0}$.

The normalized parabolically induced representation $\Ind_P^{\Mp(W)}(\pi_{M,\bfs})$ is realized on the space of $\scrV_{\tau_1} \otimes \cdots \otimes \scrV_{\tau_m} \otimes \scrV_{\pi_0}$-valued smooth functions $\scrF_\bfs$ on $\Mp(W)$ such that
\begin{equation*}
\scrF_\bfs(mng)
=\delta_P(m)^{\frac{1}{2}} \pi_{M, \bfs}(m) \scrF_\bfs(g),
\end{equation*}
for any $m \in M$, $n \in N$, and $g \in \Mp(W)$, where $\delta_P$ is the modulus function.
For $w \in \wl(\hat{M}, \Sp_{2n}(\C))$, we shall define the unnormalized intertwining operator
\begin{equation*}
\calM\left(\wtil{w}_P, \pi_{M,\bfs}\right)
: \Ind_P^{\Mp(W)}(\pi_{M,\bfs}) \to \Ind_P^{\Mp(W)}(w\pi_{M,\bfs}),
\end{equation*}
by the meromorphic continuation of the integral
\begin{equation*}
\calM\left(\wtil{w}_P,\pi_{M,\bfs}\right) \scrF_\bfs(g)
= \int_{(\prescript{w}{}{N} \cap N) \backslash N} \scrF_\bfs((\wtil{w}_P)\inv ng)dn,
\end{equation*}
where $\prescript{w}{}{N}=\wtil{w}_P N \wtil{w}_P\inv$, and $w\pi_{M,\bfs}$ is the representation of $M$ on $\scrV_{\tau_1} \otimes \cdots \otimes \scrV_{\tau_m} \otimes \scrV_{\pi_0}$ given by
\begin{equation*}
w\pi_{M,\bfs}(m)
= \pi_{M,\bfs}((\wtil{w}_P)\inv m \wtil{w}_P),
\end{equation*}
for $m \in M$.
The integral above converges absolutely on some open set of $\C^m$ in $\bfs$, and has meromorphic continuation to $\bfs \in \C^m$.
The operator is well-defined for $\bfs \in \C^m$ except finite poles modulo $(2\pi i/\log q_F)\Z^m$.
Similarly, we can define the unnormalized intertwining operator $\calM(\wtil{w}_Q, \sigma_{L,\bfs})$ from $\Ind_Q^{\SO(V)}(\sigma_{L,\bfs})$ to $\Ind_Q^{\SO(V)}(w\sigma_{L,\bfs})$.\\

Before stating the definition of the normalized self-intertwining operators, we need to normalize the operators $\calM\left(\wtil{w}_P, \pi_{M,\bfs}\right)$ and $\calM\left(\wtil{w}_Q, \sigma_{L,\bfs}\right)$ to be holomorphic at $\bfs=0$.
Put $P^w = (\wtil{w}_P)\inv P \wtil{w}_P$, $Q^w = (\wtil{w}_Q)\inv Q \wtil{w}_Q$, and let $\phi_1, \ldots, \phi_m, \phi_0$ be the $L$-parameters corresponding to $\tau_1, \ldots, \tau_m, \sigma_0$ via LLC.
Since $\sigma_0 = \Theta_\psi(\pi_0)$, the $L$-parameter for $\pi_0$ is also $\phi_0$.
As in \eqref{phim}, put
\begin{equation*}
\phi
= \phi_1 \oplus \cdots \oplus \phi_m \oplus \phi_0 \oplus \phi_m^\vee \oplus \cdots \oplus \phi_1^\vee,
\end{equation*}
so that $\phi \in \Phitemp(\SO_{2n+1}) = \Phitemp(\Mp_{2n})$ and $\myim(\phi) \subset \hat{M}$.
By the following equation, we define the twist $\phi_\bfs$ of $\phi$ by $\bfs$:
\begin{equation*}
\phi_\bfs
= (\phi_1 \otimes |-|^{s_1}) \oplus \cdots \oplus (\phi_m \otimes |-|^{s_m})
         \oplus \phi_0
          \oplus (\phi_m^\vee \otimes |-|^{-s_m}) \oplus \cdots \oplus (\phi_1^\vee \otimes |-|^{-s_1}).
\end{equation*}
Then we shall define a normalizing factor
\begin{equation*}
r_{P^w|P}(\bfs,\phi, \psi)
= \gamma(0, \rho_{P^w|P}^\vee \circ \phi_{\bfs}, \psi)\inv,
\end{equation*}
where $\rho_{P^w|P}$ denotes the representation of $\hat{M}$ defined in \cite[pp.80-81]{art13}.
Similarly, define a normalizing factor $r_{Q^w|Q}(\bfs,\phi, \psi) = \gamma(0, \rho_{Q^w|Q}^\vee \circ \phi_{\bfs}, \psi)\inv$.
The realization of $\wl(\hat{M}, \Sp_{2n}(\C))$ as a subgroup of $\wlbc{m}$ gives us an expression
\begin{equation*}
w
= \sigma_w \ltimes (d_i)_{i=1}^m.
\end{equation*}
Let $y : \Z/2\Z \to \{0,1\}$ be a map such that $y(2\Z) = 0$ and $y(1+2\Z) = 1$.
We then define a representation $y(w, \phi_\bfs)$ of $\WD_F$ and a complex number $y(w, \bfs)$ by
\begin{align*}
&y(w, \phi_\bfs) = \bigoplus_{i=1}^m y(d_i) \phi_i\otimes|-|^{s_i},&
&y(w, \bfs) = \sum_{i=1}^m y(d_i) s_i.&
\end{align*}
Let us define normalized intertwining operators
\begin{align*}
\calR_P(w, \pi_{M, \bfs}, \psi)
&= \gf(\psi)^{\dim y(w,\phi)} |2|_F^{2y(w,\bfs)} \gamma(\tfrac{1}{2}, y(w, \phi_\bfs), \psi)\inv
           r_{P^w|P}(\bfs,\phi, \psi)\inv \calM\left(\wtil{w}_P,\pi_{M,\bfs}\right), \\
\calR_Q(w, \sigma_{L, \bfs})
&= \epsilon(V)^{\dim y(w, \phi)}
          r_{Q^w|Q}(\bfs,\phi, \psi)\inv \calM\left(\wtil{w}_Q,\sigma_{L,\bfs}\right).
\end{align*}
It is known that $\calR_Q(w, \sigma_{L, \bfs})$ is independent of the choice of the additive character $\psi$. (See \cite[p.83]{art13}.)
We can see that the intertwining operators can be defined at $\bfs=0$:
\begin{lem}
The normalized intertwining operators $\calR_P(w, \pi_{M, \bfs}, \psi)$, $\calR_Q(w, \sigma_{L, \bfs})$ are holomorphic at $\bfs=0$.
\end{lem}
\begin{proof}
By \cite[Proposition 2.3.1]{art13}, $\calR_Q(w, \sigma_{L, \bfs})$ is holomorphic at $\bfs=0$ if $V=V^+$.
Next let us consider the metaplectic case.
By the definition of the representative $\wtil{w}_P$, we can decompose the operator into the product of the operators for simple reflections in $\wl(\hat{M}, \Sp_{2n}(\C))$.
Now, there are two cases: $w \in \mathfrak{S}_m$ or $w \in (\Z/2\Z)^m$.
If $w \in \mathfrak{S}_m$, the assertion is reduced to the case of $\GL_k$.
In this case the assertion follows from \cite[Proposition 3.1.4 and (3.2.1)]{sha}.
If $w \in (\Z/2\Z)^m$, the assertion is reduced to the case of $m=1$, i.e., $P=P_\bfk$ is a maximal parabolic subgroup $P_k$.
In this case, since the explicit formula of Plancherel measures for the metaplectic group (\cite[Appendix A.7]{gigp}) is known, the assertion can be proven as in \cite[Theorem 2.1]{artjfa}.
In case $V=V^-$, a similar argument goes.
\end{proof}

We shall put
\begin{align*}
&\calR_P(w, \pi_M, \psi) = \calR_P(w, \pi_{M,0}, \psi),&
&\calR_Q(w, \sigma_L) = \calR_Q(w, \sigma_{L,0}).&
\end{align*}
Let us define the normalized self-intertwining operators.
Assume
\begin{equation*}
w \in \wl_\phi(M,\Mp(W)) = \wl_\phi(L,\SO(V)),
\end{equation*}
which is equivalent to $w \pi_M \cong \pi_M$ and $w \sigma_L \cong \sigma_L$.
We take the unique Whittaker normalized isomorphism
\begin{equation*}
\calA_w : \scrV_{\tau_1} \otimes \cdots \otimes \scrV_{\tau_m}
 \lra \scrV_{\tau_1} \otimes \cdots \otimes \scrV_{\tau_m},
\end{equation*}
and define the normalized self-intertwining operators
\begin{align*}
R_P(w, \pi_M)
&: \Ind_P^{\Mp(W)}(\pi_M) \lra \Ind_P^{\Mp(W)}(\pi_M),\\
R_Q(w, \sigma_L)
&: \Ind_Q^{\SO(V)}(\sigma_L) \lra \Ind_Q^{\SO(V)}(\sigma_L)
\end{align*}
as in \cite[p.756]{gigp}.
%%%%%%%%%%%%%%%%%%%%%%%%%%%%%%%%%%%%%%%%%%%%%%%%%%%%%%%%%%%%%%%%%%%%%%%%%%%%%%%%%%%%%%%%%
\subsection{Reduction}\label{red}
Since the LLC for $\Mp(W)$ is defined by using the theta correspondence, it suffices for proving Theorem \ref{lirmp} to consider the relation between the theta correspondence and the intertwining operators.
The following proposition will be proven later.
\begin{prop}\label{reduction}
Put $\breve{\sigma}_L = \tau_1 \otimes \cdots \otimes \tau_m \otimes \sigma_0^\vee$.
There exists a nonzero $\SO(V) \times \Mp(W)$-equivariant map
\begin{align*}
\calT : \omega_{V,W,\psi} \otimes \Ind_Q^{\SO(V)}(\breve{\sigma}_L) \lra \Ind_P^{\Mp(W)}(\pi_M)
\end{align*}
such that
\begin{enumerate}[(a)]
\item
for any irreducible constituent $\sigma$ of $\Ind_Q^{\SO(V)}(\breve{\sigma}_L)$, the restriction of $\calT$ to $\omega_{V,W,\psi} \otimes \sigma$ is nonzero;
\item
the diagram
\begin{align*}
\begin{CD}
               \omega_{V,W,\psi} \otimes \Ind_Q^{\SO(V)}(\breve{\sigma}_L) @>\calT>> \Ind_P^{\Mp(W)}(\pi_M)  \\
               @V 1_{\omega_{V,W,\psi}} \otimes R_Q(w, \breve{\sigma}_L) VV               @VV R_P(w, \pi_M) V  \\
               \omega_{V,W,\psi} \otimes \Ind_Q^{\SO(V)}(\breve{\sigma}_L) @>\calT>> \Ind_P^{\Mp(W)}(\pi_M)
\end{CD}
\end{align*}
commutes.
\end{enumerate}
\end{prop}
Once the proposition is proven, we have Theorem \ref{lirmp}:
\begin{prop}
Proposition \ref{reduction} implies Theorem \ref{lirmp}.
\end{prop}
\begin{proof}
Suppose that Proposition \ref{reduction} holds.
Because any irreducible representation of an odd special orthogonal group is self-dual (\cite[Chapter 4. II. 1]{mvw}), we have $\sigma_0^\vee\cong\sigma_0$.
Now fix an isomorphism $\sigma_L \cong \breve{\sigma}_L$ and identify them.

Let $\pi \subset \Ind_P^{\Mp(W)}(\pi_M)$ be an irreducible tempered representation and put $\sigma = \Theta_\psi(\pi) \subset \Ind_Q^{\SO(V)}(\sigma_L)$.
Then by Proposition \ref{reduction} and the identification $\sigma_L \cong \breve{\sigma}_L$, there exists a nonzero $\SO(V) \times \Mp(W)$-equivariant map
\begin{equation*}
\calT : \omega_{V,W,\psi} \otimes \Ind_Q^{\SO(V)}(\sigma_L) \lra \Ind_P^{\Mp(W)}(\pi_M),
\end{equation*}
such that its restriction to $\omega_{V,W,\psi} \otimes \sigma$ is nonzero and it satisfies the following commutative diagram:
               \begin{align*}
                \begin{CD}
                 \omega_{V,W,\psi} \otimes \Ind_Q^{\SO(V)}(\sigma_L) @>\calT>> \Ind_P^{\Mp(W)}(\pi_M)  \\
                 @V 1_{\omega_{V,W,\psi}} \otimes R_Q(w, \sigma_L) VV               @VV R_P(w, \pi_M) V  \\
                 \omega_{V,W,\psi} \otimes \Ind_Q^{\SO(V)}(\sigma_L) @>\calT>> \Ind_P^{\Mp(W)}(\pi_M).
                \end{CD}
               \end{align*}

By the fact that $\sigma^\vee \cong \sigma$ and the Howe duality, $\calT$ sends $\omega_{V,W,\psi} \otimes \sigma$ to $\pi$.
Therefore, $\calT$ gives a nonzero $\SO(V) \times \Mp(W)$-equivariant map
\begin{equation*}
\calT_{\sigma, \pi} : \omega_{V,W,\psi} \otimes \sigma \lra \pi
\end{equation*}
such that
\begin{equation}\label{redeq}
R_P(w, \pi_M)|_\pi \circ \calT_{\sigma, \pi}
=\calT_{\sigma, \pi} \circ (1_{\omega_{V,W,\psi}} \otimes R_Q(w, \sigma_L)|_\sigma).
\end{equation}

Suppose that Hypothesis \ref{lirso} holds.
Then we have $R_Q(w, \sigma_L)|_\sigma = \eta(x_w)$, where $\eta = \iota(\sigma)$, and $x_w \in S_\phi^\natural(L, \SO(V)) = S_\phi^\natural(M, \Mp(W))$ is the image of $w$ under the natural map \eqref{x}.
Now the relation \eqref{redeq} shows that
\begin{equation*}
R_P(w, \pi_M)|_\pi \circ \calT_{\sigma, \pi}
=\eta(x_w) \calT_{\sigma, \pi}.
\end{equation*}
Since $\calT_{\sigma, \pi} \neq 0$ and $\pi$ is irreducible, we have $R_P(w, \pi_M)|_\pi = \eta(x_w)$.
We also have $\eta = \iota_\psi(\pi)$, by the definition of LLC for $\Mp(W)$.
This completes the proof.
\end{proof}
%%%%%%%%%%%%%%%%%%%%%%%%%%%%%%%%%%%%%%%%%%%%%%%%%%%%%%%%%%%%%%%%%%%%%%%%%%%%%%%%%%%%%%%%%%%%%%
%%%%%%%%%%%%%%%%%%%%%%%%%%%%%%%%%%%%%%%%%%%%%%%%%%%%%%%%%%%%%%%%%%%%%%%%%%%%%%%%%%%%%%%%%%%%%%
\section{Preparations for the proof of Proposition \ref{reduction}}
In the next section, we shall give a proof of Proposition \ref{reduction}.
For this, we introduce some more notation following Gan-Ichino \cite[\S7, \S8]{gigp}, in this section.
\subsection{Maximal parabolic subgroups}\label{maxparab}
We have described the parabolic subgroups of $\Sp(W)$, $\Mp(W)$, and $\SO(V)$ in \S\ref{symp}, \S\ref{mp}, and \S\ref{orth}, respectively.
Referring to \cite{ato}, we can describe their maximal parabolic subgroups more explicitly.

Let $k$ be a positive integer, and put $n_0=n-k$.
Put $Y= Y_k$, $Y^*= Y^*_k$.
We shall write an element in the symplectic group $\Sp(W)$ as a block matrix relative to the decomposition $W=Y \oplus W_{n_0} \oplus Y^*$.
Following \S \ref{symp} or \S \ref{mp}, put $P=P_k$, $M=M_k$, and $N=N_k$, so that $\mpsp{P}=\mpsp{P_k}$ and $\mpsp{M}=\mpsp{M_k}$.
Then we have
\begin{align*}
\mpsp{M}
&=\Set{m(a)g_0 | a \in \GL(Y),\ g_0 \in \Sp(W_{n_0}) }, \\
N
&=\Set{ n^\rmb(b) n^\rmc(c) | b \in \Hom(W_{n_0},Y),\  c \in \Sym(Y^*,Y) },
\end{align*}
where $m(a)$, $n^\rmb(b)$, $n^\rmc(c)$, and $\Sym(Y^*,Y)$ are defined as in \cite[\S2.4]{ato}.
Recall that $P$ and $M$ are the double covers of $\mpsp{P}$ and $\mpsp{M}$, respectively.
Note that the natural inclusion $\Sp(W_{n_0}) \subset \Sp(W)$ induces an inclusion $\Mp(W_{n_0}) \subset \Mp(W)$, $(g_0,\epsilon) \mapsto (g_0,\epsilon)$.
Put
\begin{equation*}
\rho_P=\frac{2n-k+1}{2}.
\end{equation*}

Assume that $k\leq r$.
Put $X=X_k$, $X^*=X^*_k$ and we shall write an element in the special orthogonal group $\SO(V)$ as a block matrix relative to the decomposition $V=X\oplus V_{n_0} \oplus X^*$, as above.
Put $Q=Q_k$, $L=L_k$, and $U=U_k$, following \S \ref{orth}.
Then we have
\begin{align*}
L
&=\Set{l(a)h_0 | a \in \GL(X),\ h_0 \in \SO(V_{n_0})}, \\
U
&= \Set{ u = u^\rmb(b) u^\rmc(c) | b \in \Hom(V_{n_0},X),\ c \in \Alt(X^*,X) },
\end{align*}
where $l(a)$, $u^\rmb(b)$, $u^\rmc(c)$, and $\Alt(X^*,X)$ are given in a similar way to \cite[\S2.4]{ato}.
Put
\begin{equation*}
\rho_Q=\frac{2n-k}{2}.
\end{equation*}
%%%%%%%%%%%%%%%%%%%%%%%%%%%%%%%%%%%%%%%%%%%%%%%%%%%%%%%%%%%%%%%%%%%%%%%%%%%%%%%%%%%%%%%%%
\subsection{Representatives of $w_M$ and $w_L$}
Let $w_M$ (resp. $w_L$) be the nontrivial element of the relative Weyl group $\wl(\mpsp{M}, \Sp(W))$ (resp. $\wl(L, \SO(V))$).
Note that $\wl(\mpsp{M},\Sp(W)) \cong \wl(L, \SO(V)) \cong \Z/2\Z$.
In this subsection, we shall take representatives of $w_M$ and $w_L$ following Langlands-Shelstad (see \S\ref{weyl}), and calculate them explicitly.

First, let us define $I_X \in \Hom(X^*,X)$ and $I_Y \in \Hom(Y^*,Y)$ by $I_X x_i^* = x_i$ and $I_Y y_i^* = y_i$.
With respect to the bases, $I_X$ and $I_Y$ correspond to the identity matrix.
Put
\begin{align*}
J
&=\left(\begin{array}{cccc}
                  &&&(-1)^{n+1} \\
                  &&(-1)^{n+2}& \\
                  &\adots&& \\
                  (-1)^{n+k}&&& \end{array} \right) \in \GL_k.
\end{align*}
Using the bases, we can identify $\GL(X)$ and $\GL(Y)$ with $\GL_k$, and consider $J$ as an element of $\GL(X)$ or $\GL(Y)$.
Let us define elements $w_Y \in \Sp(W)$ and $w_X \in \SO(V)$ by
\begin{align*}
&w_Y = \left(\begin{array}{ccc}
                        &&I_Y \\
                        &(-1)^k1_{W_{n_0}}&\\
                        -I_Y^{-1}&&\end{array}\right),&
&w_X = \left(\begin{array}{ccc}
                      &&-I_X      \\
                      &(-1)^k1_{V_{n_0}}&\\
                      -I_X^{-1}&&       \end{array}\right).&
\end{align*}
We take the representatives $w_M' \in \Mp(W)$ and $w_L' \in \SO(V)$ of $w_M$ and $w_L$ defined by
\newcommand{\LS}{\mathrm{LS}}
\begin{align*}
&w_M'
=\left((-1)^k\begin{pmatrix}&&-J I_Y \\ &1_{W_{n_0}}& \\ J I_Y^{-1}&&\end{pmatrix}, \epsilon^\LS\right),&
&w_L'
=(-1)^k\begin{pmatrix}&&J I_X \\ &1_{V_{n_0}}& \\ J I_X^{-1}&&\end{pmatrix},&
\end{align*}
respectively, where $\epsilon^\LS = (-1,-1)_F^{k(k-1)/2}$.

Let $\wtil{w}_P$ and $\wtil{w}_Q$ denote Langlands-Shelstad's representatives (\cite[\S2.1]{ls} and \cite[Definition 4.1]{ganli}) of $w_M$ and $w_L$ with respect to the $F$-splittings $\splw$ and $\splv$.
Then we have the following proposition:
\begin{prop}\label{LS}
We have
\begin{align*}
\wtil{w}_P &= w'_M, \\
\wtil{w}_Q &= w'_L.
\end{align*}
\end{prop}
The proof is similar to that of \cite[Lemma 7.2]{gigp}.
Also, as pointed out in \cite[p.755]{gigp}, in the case $V = V^-$, one can see that $w_L'$ corresponds to $\wtil{w}_Q \in \SO(V^+)$ via the canonical pure inner twist.
However, we shall give a proof of the first assertion in \S\ref{pfLS}, since the calculation of $\wtil{w}_P$ is too complicated because we have to consider Ranga Rao's 2-cocycle.
%%%%%%%%%%%%%%%%%%%%%%%%%%%%%%%%%%%%%%%%%%%%%%%%%%%%%%%%%%%%%%%%%%%%%%%%%%%%%%%%%%%%%%%%%
\subsection{Ranga Rao's 2-cocycle}
Before the proof of Proposition \ref{LS}, we introduce some notation and review Ranga Rao's $x$-function and Ranga Rao's normalized cocycle \cite{rao} here.

For three nonnegative integers $r, s, t \in \Z_{\geq 0}$, define $\iota^{\GL}_{r,s,t}$ to be an embedding of $\GL_s$ into $\GL_{r+s+t}$ by
\begin{align*}
    A \mapsto \left( \begin{array}{ccc}
                            1_r&& \\
                            &A&   \\
                            &&1_t  \end{array} \right).
\end{align*}
For $a \in \GL(Y_n)$, we shall write $m_n(a)$ for an element $\mmatrix{a}{}{}{(a^*)\inv}$ of a Levi subgroup $\mpsp{M_n}$ of the Siegel parabolic subgroup $\mpsp{P_n}$.
For any subset $S \subset \set{1,\dots,n }$, define $\sigma_S$ and $a_S$ by
\begin{align*}
&\sigma_S \cdot y_i=\begin{cases} y_i^* & \text{if } i\in S, \\
                                                  y_i &\text{if }i\notin S, \end{cases}&
&\sigma_S \cdot y_i^*=\begin{cases} -y_i &\text{if } i\in S, \\
                                                  y_i^*& \text{if }i\notin S,  \end{cases}&
\end{align*}
and
\begin{align*}
&a_S \cdot y_i=\begin{cases}     -y_i & \text{if } i\in S, \\
                                             y_i &\text{if }i\notin S, \end{cases}&
&a_S \cdot y_i^*=\begin{cases}  -y_i^* &\text{if } i\in S, \\
                                               y_i^*& \text{if }i\notin S.  \end{cases}&
\end{align*}
When $S$ is a singleton $\{i\}$, we shall write $\sigma_i = \sigma_{\{i\}}$, for simplicity.

Next, we review the notion of Ranga Rao's $x$-function and Ranga Rao's normalized cocycle.
We have $\Sp(W) = \cup \mpsp{P_n} \sigma_S \mpsp{P_n}$, where the disjoint union runs over all subset $S \subset \{1, \ldots, n\}$, and Ranga Rao's $x$-function is defined by
\begin{equation*}
x(p_1 \sigma_S p_2) = \det(p_1p_2|_{Y_n}) \left( \mathrm{mod} (F^\times)^2 \right),
\quad
p_1, p_2 \in \mpsp{P_n}, \ S \subset \set{1, \ldots, n}.
\end{equation*}
This is well-defined (\cite[Lemma 5.1]{rao}).
Then, let $c(-,-)$ denote Ranga Rao's normalized cocycle, which is a 2-cocycle on $\Sp(W)$ valued in $\{\pm1\}$.
The precise definition of $c(-,-)$ is omitted here, but we shall list several of its properties.
See \cite[\S5]{rao} or \cite[\S2]{szp} for detail.
\begin{prop}
Let $p, p' \in \mpsp{P_n}$, $S, S' \subset \set{1, \ldots, n}$, and $g, g' \in \Sp(W)$.
Put $j = |S \cap S'|$.
Then we have
\begin{align*}
c(\sigma_S, \sigma_{S'})
&= (-1,-1)_F^{\frac{j(j+1)}{2}}, \\
c(pg, g'p')
&= c(g, g') (x(g), x(p))_F (x(g'), x(p'))_F (x(p), x(p'))_F (x(gg'), x(pp'))_F, \\
c(g,p)
&= c(p,g) = (x(p), x(g))_F.
\end{align*}
Moreover if $gg' = g'g$, then we have
\begin{equation*}
c(g,g') = c(g',g).
\end{equation*}
\end{prop}
%%%%%%%%%%%%%%%%%%%%%%%%%%%%%%%%%%%%%%%%%%%%%%%%%%%%%%%%%%%%%%%%%%%%%%%%%%%%%%%%%%%%%%%%%
\subsection{Proof of Proposition \ref{LS}}\label{pfLS}
Now we shall begin the proof of the first assertion of Proposition \ref{LS}.
\begin{proof}
First, we need a certain representative of $w_M$ in $\wl(T_W,\Sp(W))$.
Take the representative $w_T \in \wl(T_W,\Sp(W))$ of $w_M$ such that $w_T$ maps the positive roots inside $\mpsp{M}$ into the positive roots inside $\mpsp{M}$, the positive roots outside $\mpsp{M}$ into the negative roots (not necessarily outside $\mpsp{M}$).
Let $s_i \in \wl(T_W,\Sp(W))$ be the simple reflection corresponding to $\alpha_i \in \Delta(B_W)$, and put
\begin{align*}
q_i
&= s_{k-1}s_{k-2}\cdots s_{i+1}s_i &  &\text{for } 1\leq i\leq k-1,\\
r_i
&= s_is_{i+1}\cdots s_{n-1}s_ns_{n-1}\cdots s_{i+1}s_i & &\text{for } 1\leq i\leq n.
\end{align*}
Then
\begin{equation*}
w_T = r_kq_1r_kq_2r_k\cdots q_{k-2}r_kq_{k-1}r_k
\end{equation*}
gives a reduced decomposition of $w_T$. \\

Second, let us consider the representative $\wtil{w}_{\mpsp{P}}$ of $w_T$ in the symplectic group $\Sp(W)$, following  Langlands-Shelstad \cite{ls}.
Put
\begin{align*}
\omega_i &= \exp(X_{\alpha_i})\exp(-X_{-\alpha_i})\exp(X_{\alpha_i})&   &\text{for }1\leq i\leq n,  \\
u_i &= \omega_{k-1}\cdots \omega_{i+1}\omega_i &    &\text{for }1\leq i\leq k-1, \\
v_i &= \omega_i\cdots \omega_{n-1}\omega_n\omega_{n-1}\cdots \omega_i &   &\text{for } 1\leq i\leq n,
\end{align*}
which are representatives of $s_i,q_i,r_i$, respectively.
Then by \cite[\S2.1]{ls}, we have the representative $\wtil{w}_{\mpsp{P}}$ in $\Sp(W)$:
\begin{equation*}
\wtil{w}_{\mpsp{P}} = v_k u_1v_k u_2\cdots v_k u_{k-1}v_k.
\end{equation*}
In addition, put
\begin{align*}
z_i &= \omega_{k-1}\inv \cdots \omega_{i+1}\inv \omega_i\inv& &\text{for } 1\leq i\leq k-1.
\end{align*}
Then we obtain that $v_k u_i = z_iv_i$ for $1\leq i\leq k-1$.
Moreover, one can calculate $v_i$ and $z_j$ by descending induction on $i=n,\dots,1$ and $j=k-1,\dots,1$ and obtain
\begin{align*}
v_i &= a_{\{i+1, \ldots, n\}} \sigma_i^{2(n-i-1)+1}, \\
z_j &= m_n(\iota^{\GL}_{j-1, k-j+1, n-k}(\kappa_{k-j+1})),
\end{align*}
where
\begin{align*}
\kappa_l
=\left(\begin{array}{cccc}
                    0&-1&&\\
                    &&\ddots&\\
                    &&&-1\\
                    1&&&0\\
                  \end{array}\right) \in \GL_l.
\end{align*}
Then a straightforward calculation shows that $v_iz_j = z_jv_i$, for $1\leq i<j\leq k-1$.
This implies that $\wtil{w}_{\mpsp{P}}=z_1\cdots z_{k-1}v_1\cdots v_k$.\\

Finally, let us take their representatives in $\Mp(W)$ as follows.
Put
\begin{align*}
\wtil{\omega}_i &= \wtil{w}_{\alpha_i} = (\omega_i,1) & &\text{for } 1\leq i\leq n, \\
\wtil{u}_i &= \wtil{\omega}_{k-1}\cdots\wtil{\omega}_{i+1}\wtil{\omega}_i & &\text{for } 1\leq i\leq k-1, \\
\wtil{v}_i &=\wtil{\omega}_i\cdots\wtil{\omega}_{n-1}\wtil{\omega}_n\wtil{\omega}_{n-1}\cdots\wtil{\omega}_i & &\text{for } 1\leq i\leq n,
\end{align*}
and
\begin{align*}
\wtil{z}_i &= \wtil{\omega}_{k-1}\inv \cdots \wtil{\omega}_{i+1}\inv \wtil{\omega}_i\inv & &\text{for } 1\leq i\leq k-1.
\end{align*}
Then the required element $\wtil{w}_P$ can be expressed as
\begin{equation*}
\wtil{w}_P = \wtil{v}_k\wtil{u}_1\wtil{v}_k\wtil{u}_2\cdots\wtil{v}_k\wtil{u}_{k-1}\wtil{v}_k.
\end{equation*}
We have $\wtil{v}_k\wtil{u}_i=\wtil{z}_i\wtil{v}_i$ for $1\leq i\leq k-1$.
Also, for $1\leq i<j\leq k-1$ we have $\wtil{v}_i\wtil{z}_j = \wtil{z}_j \wtil{v}_i$ because $v_iz_j = z_jv_i$.
Therefore,
\begin{align*}
\wtil{w}_P
&=\wtil{z}_1\wtil{v}_1\wtil{z}_2\wtil{v}_2\cdots\wtil{z}_{k-1}\wtil{v}_{k-1}\wtil{v}_k \\
&=\wtil{z}_1\wtil{z}_2\cdots\wtil{z}_{k-1}\wtil{v}_1\wtil{v}_2\cdots\wtil{v}_{k-1}\wtil{v}_k.
\end{align*}
Since $\omega_1,\dots, \omega_{n-1}$ are elements of the Siegel parabolic subgroup $\mpsp{P_n}$ and have determinant 1 on $Y_n$, one has
\begin{align*}
\wtil{v}_i  &= (v_i,1) & \text{for }& 1\leq i \leq k,\\
\wtil{z}_j &= (z_j,1) & \text{for }& 1 \leq j \leq k-1.
\end{align*}

Now let us compute $\wtil{z}_1\wtil{z}_2\cdots\wtil{z}_{k-1}$, $\wtil{v}_1\wtil{v}_2\cdots\wtil{v}_{k-1}\wtil{v}_k$, and $\wtil{w}_P$.
First, we shall consider $\wtil{z}_1\wtil{z}_2\cdots\wtil{z}_{k-1}$.
Since each $z_j$ belongs to the Siegel parabolic subgroup and has determinant $1$ on $Y_n$, we have $\wtil{z}_1\cdots\wtil{z}_{k-1}=(z_1\cdots z_{k-1},1)$.
Additionally, calculating its action on the basis $\{y_1,\dots,y_n,y_1^*,\dots,y_n^* \}$, we can compute the product $z_1\cdots z_{k-1}$:
\begin{equation*}
z_1\cdots z_{k-1}
= m_n(\iota^{\GL}_{0,k,n-k}((-1)^{n+k}J)).
\end{equation*}
Second, by descending induction, we can compute $\wtil{v}_i \cdots \wtil{v}_k$ for $i=k,\dots,1$.
Note that Ranga Rao's normalized cocycle may not be trivial.
By descending induction, one has
\begin{equation*}
v_i\cdots v_k
= \sigma_{\{i,\dots,k\}}p'_i,
\end{equation*}
where
\begin{equation*}
p_i'
= (a_{\{i, \ldots, k\}})^{n-i+1} (a_{\{k+1, \ldots, n\}})^{k-i+1}.
\end{equation*}
We also have
\begin{equation*}
v_i
= p_i\sigma_i,
\end{equation*}
where
\begin{equation*}
p_i
= (a_{\{i\}})^{n-i+1} a_{\{i+1, \ldots, n\}}.
\end{equation*}
Since $p_i'$ and $p_i$ are elements of the Siegel parabolic subgroup, one has
\begin{align*}
c(v_i,v_{i+1}\cdots v_k)
&=c(p_i\sigma_{i},\sigma_{\{i+1,\dots,k\}}p_{i+1}') \\
&=(x(p_i),x(p_{i+1}'))_F \\
&=((-1)^{(n-i+1)+(n-i)},(-1)^{(n-i)(k-i)+(k-i)(n-k)})_F \\
&=(-1,-1)_F^{k+i}.
\end{align*}
Hence,
\begin{align*}
\wtil{v}_1\cdots\wtil{v}_k
&=(v_1\cdots v_k, \prod_{i=1}^{k-1} c(v_i,v_{i+1}\cdots v_k)) \\
&=\left( (\sigma_{\{1, \ldots, k\}})^{2n+1} (a_{\{k+1, \ldots, n\}})^k, (-1,-1)_F^{\frac{k(k-1)}{2}} \right).
\end{align*}

Finally, since $z_1\cdots z_{k-1}$ belongs to the Siegel parabolic subgroup and has determinant $1$ on $Y_n$, we have $\wtil{w}_P= (\wtil{w}_{\mpsp{P}}, \epsilon^\LS)$ and $\wtil{w}_{\mpsp{P}}= z_1\cdots z_{k-1}v_1\cdots v_k  = (-1)^km_n(\iota^{\GL}_{0,k,n-k}(J)) \sigma_{\{1, \ldots, k\}}$.
This is the first assertion of the proposition.
\end{proof}
%%%%%%%%%%%%%%%%%%%%%%%%%%%%%%%%%%%%%%%%%%%%%%%%%%%%%%%%%%%%%%%%%%%%%%%%%%%%%%%%%%%%%%%%%
\subsection{Haar measures}\label{measm}
In order to study the intertwining operators in more detail, or to describe some explicit formulas for the Weil representations, we need to take Haar measures appropriately and explicitly.
Put
\begin{align*}
e     &=x_1\otimes y^*_1+\dots+x_k\otimes y_k^* \in X\otimes Y^*, \\
e^*  &=x^*_1\otimes y_1+\dots+x^*_k\otimes y_k \in X^*\otimes Y, \\
e^{**}&=x^*_1\otimes y^*_1+\dots+x^*_k\otimes y^*_k \in X^*\otimes Y^*.
\end{align*}
These vectors belong to the symplectic space $\W=V \otimes_F W$.

Let us define measures on each groups and vector spaces.
\begin{enumerate}
   \item Take the self-dual Haar measure $d_{\rmM_k}x$ on $\rmM_k(F)$ with respect to the pairing
             \begin{align*}
              \rmM_k(F) \times \rmM_k(F) \ni (x,y)\mapsto \psi(\mathrm{tr}(xy)) \in \C^1.
             \end{align*}
            In particular, let us write $d_\psi x$ when $k=1$.
   \item Take the Haar measure $dx$ on $\GL_k(F)$ defined by $dx = |\det x|_F^{-k} d_{\rmM_k}x$,
             and we transfer it to $\GL(Y)$ and $\GL(X)$ via the identification.
   \item  Define the self-dual Haar measures
             on $V\otimes Y^*$, $X^*\otimes Y$, $X\otimes Y^*$, $V_{n_0}\otimes Y^*$, $X^*\otimes W_{n_0}$,
              $\Hom(V_{n_0},X)$, $\Hom(W_{n_0},Y)$, $\Hom(X,X)$, and $\Hom(Y,Y)$
              in a similar way to \cite[\S 7.2]{gigp}.
   \item  Take the self-dual Haar measures on $\Alt(X^*,X)$ and $\Sym(Y^*,Y)$ with respect to the pairings
              \begin{align*}
               \Alt(X^*,X) \times \Alt(X^*,X) \ni (c,c')
                    &\mapsto \psi(\an{ I_Yce^{**}, I_X\inv c'e^{**} }) \in \C^1,\\
               \Sym(Y^*,Y) \times \Sym(Y^*,Y) \ni (c, c')
                   &\mapsto \psi(\an{ I_Xce^{**}, I_Y\inv c'e^{**} }) \in \C^1,
              \end{align*}
              respectively.
   \item  Take the Haar measures
              $du$ on $U$ for $u=u^\rmb(b)u^\rmc(c)$,
              and $dn$ on $N$ for $n=n^\rmb(b)n^\rmc(c)$,
              as follows:
               \begin{align*}
                du&=|2|_F^{-\frac{k}{2}}db \cdot |2|_F^{-\frac{k(k-1)}{4}}dc,&  &b\in\Hom(V_{n_0},X), c\in\Alt(X^*,X), \\
                dn&=|2|_F^{-\frac{k}{2}}db \cdot |2|_F^{-\frac{k(k-1)}{4}}dc,&  &b\in\Hom(W_{n_0},Y), c\in\Sym(Y^*,Y).
               \end{align*}
   \item  Let us take measures on $Q$ and $\mpsp{P}$.
            For $q=lu \in Q=LU$ and $p=mn \in \mpsp{P}=\mpsp{M} N$, we define
             \begin{align*}
              dq&=dldu& &\text{and}& dp&=dmdn.
             \end{align*}
            We have the modulus functions
             $\delta_Q(l(a)h_0)=|\det a|_F^{2\rho_Q}$ for $a\in \GL(X), h_0\in \SO(V_{n_0})$,
            and $\delta_P(m(a)g_0)=|\det a|_F^{2\rho_P}$ for $a\in \GL(Y), g_0\in \Sp(W_{n_0})$.
\end{enumerate}
One can then check that the measures $du$ on $U$ and $dn$ on $N$ coincide with the Haar measures that we took in \S\ref{measu} by using the splittings $\splv$ and $\splw$, respectively.
(See \cite[\S6.3]{ato} for explicit calculation.)
%%%%%%%%%%%%%%%%%%%%%%%%%%%%%%%%%%%%%%%%%%%%%%%%%%%%%%%%%%%%%%%%%%%%%%%%%%%%%%%%%%%%%%%%%
\subsection{Big symplectic spaces and a mixed model}\label{mixmodel}
In this subsection, we shall take a mixed model, which is a realization of the Weil representation, following Gan-Ichino \cite[\S7.4]{gigp}.

Put $\W_0=V\otimes W_{n_0}\subset \W$ and $\W_{00}=V_{n_0}\otimes W_{n_0} \subset \W_0 \subset \W$.
These are symplectic subspaces of $\W$.
Fix a polarization $W_{n_0}=W_{01}\oplus W_{02}$, where $W_{01} = \myspan_F(y_{k+1}, \ldots, y_n)$ and $W_{02} = \myspan_F(y^*_{k+1}, \ldots, y_n^*)$.
We have the following natural complete polarizations of $\W$, $\W_0$, and $\W_{00}$:
\begin{align*}
\W      &= (V\otimes Y_n)  \oplus (V\otimes Y^*_n), \\
\W_0   &= (V\otimes W_{01})  \oplus (V\otimes W_{02}), \\
\W_{00}&= (V_{n_0}\otimes W_{01})  \oplus (V_{n_0}\otimes W_{02}).
\end{align*}
Let $\omega$, $\omega_{0}$, and $\omega_{00}$ be the realizations of the Weil representations $\omega_{V,W,\psi}$, $\omega_{V, W_{n_0}, \psi}$, and $\omega_{V_{n_0}, W_{n_0}, \psi}$ of $\Or(V)\times\Mp(W)$, $\Or(V)\times \Mp(W_{n_0})$ and $\Or(V_{n_0})\times \Mp(W_{n_0})$, respectively, on a mixed Schr\"odinger model
\begin{align*}
\calS_{00}&=\calS(V_{n_0}\otimes W_{02}), \\
\calS_0&= \calS(X^*\otimes W_{0})\otimes \calS_{00}, \\
\calS&= \calS(V\otimes Y^*)\otimes\calS_0,
\end{align*}
as in \cite[\S7.4]{gigp} or \cite[\S6.2]{ato}.
We construct these models by using the following elements:
\begin{itemize}
    \item the ordinary Schr\"odinger models
             \begin{align*}
              &(\omega^{\mathrm{or}}, \calS^{\mathrm{or}}=\calS(V\otimes(Y^*\oplus W_{02}))),&
               &(\omega_0^{\mathrm{or}}, \calS_0^{\mathrm{or}}=\calS(V\otimes W_{02})),&
                &(\omega_{00}^{\mathrm{or}}, \calS_{00}^{\mathrm{or}}=\calS(V_{n_0}\otimes W_{02})),&
             \end{align*}
             of $\omega_{V,W,\psi}$, $\omega_{V, W_{n_0}, \psi}$, and $\omega_{V_{n_0}, W_{n_0}, \psi}$, respectively;
    \item canonical linear isomorphisms
             \begin{align*}
              \calS(V\otimes (Y^*\oplus W_{02})) &\cong \calS(V\otimes Y^*)\otimes \calS(V\otimes W_{02}), \\
              \calS(V\otimes W_{02}) &\cong \calS((X\oplus X^*)\otimes W_{02})\otimes \calS(V_{n_0}\otimes W_{02});
             \end{align*}
    \item an isomorphism given by the partial inverse Fourier transform
             \begin{align*}
              \calS((X\oplus X^*)\otimes W_{02}) \to \calS(X^*\otimes W_{0}), \quad
              \varphi  \mapsto  \hat{\varphi}
             \end{align*}
            defined by
             \begin{align*}
              \hat{\varphi}\left(\begin{array}{c}x_1\\x_2\end{array}\right)
               &=\int_{y\in X\otimes W_{02}}
                  \varphi\left(\begin{array}{c}y\\x_2\end{array}\right)\psi(-\an{ x_1,y })dy &
                   \text{for } x_1\in X^*\otimes W_{01},\ x_2\in X^*\otimes W_{02},
             \end{align*}
             where the Haar measure $dy$ on $X \otimes W_{02}$ is defined by
             \begin{equation*}
             dy = \prod_{\substack{1\leq i \leq k \\ k+1\leq j \leq n}} d_\psi c_{i,j},
             \quad
             \text{ for }
             \quad
             y = \sum_{\substack{1\leq i \leq k \\ k+1\leq j \leq n}} c_{i,j} x_i \otimes y_j^* \in X \otimes W_{02}.
             \end{equation*}
\end{itemize}
Let $\scrH_0 = \W_0 \oplus F$ and $\scrH_{00} = \W_{00} \oplus F$ be the Heisenberg groups.
Let $\rho_0$ and $\rho_{00}$ be their Heisenberg representations associated with the Weil representations $(\omega_0, \calS_0)$ and $(\omega_{00}, \calS_{00})$, respectively.
We consider $\Sp(W) = \Sp(W)\times \{1\} \subset \Sp(W)\times\{\pm1\} = \Mp(W)$ as sets.
Referring to \cite{rao} or \cite[Theorem 3.1]{spl}, we obtain some explicit formulas for the Weil representations.

For $\varphi\in\calS$ and $x\in V\otimes Y^*$,
\begin{align*}
  [\omega(h)\varphi](x)
    &=\omega_0(h)\varphi(h^{-1}x), &                              h&\in \SO(V), \\
  [\omega(m(a))\varphi](x)
    &=\gamma_F(\det a,\psi)^{-1}|\det a|_F^{\frac{2n+1}{2}}\varphi(a^*x), &           a&\in \GL(Y), \\
  [\omega(g_0)\varphi](x)
    &=\omega_0(g_0)\varphi(x), &                                g_0&\in \Sp(W_{n_0}), \\
  [\omega(n^{\mathrm{b}}(b))\varphi](x)
    &=\rho_0((b^*x,0))\varphi(x), &                                    b&\in\Hom(W_{n_0},Y), \\
  [\omega(n^\mathrm{c}(c))\varphi](x)
    &=\psi\left(\tfrac{1}{2}\langle n^\mathrm{c}(c)x,x\rangle\right)\varphi(x), &                  c&\in\Sym(Y^*,Y),\\
  [\omega(w_Y^{-1})\varphi](x)
    &=\gamma_F(\psi\circ q_V)^{-k}\omega_0((-1_{W_{n_0}})^k)\int_{Y^*\otimes V}\psi(\langle x',I_Yx\rangle)\varphi(x')dx'.
\end{align*}

For $\varphi_0\in\calS_0=\calS(X^*\otimes W_{n_0})\otimes\calS_{00}$ and $\ y\in X^*\otimes W_{n_0}$,
\begin{align*}
    [\omega_0(g_0)\varphi_0](y)
      &=\omega_{00}(g_0)\varphi_0(g_0^{-1}y), &                              g_0&\in \Sp(W_{n_0}), \\
    [\omega_0(l(a))\varphi_0](y)
      &=|\det a|_F^{n-k}\varphi_0(a^*y), &                                      a&\in \GL(X), \\
    [\omega_0(h_0)\varphi_0](y)
      &=\omega_{00}(h_0)\varphi_0(y), &                                   h_0&\in \SO(V_{n_0}), \\
    [\omega_0(u^\mathrm{b}(b))\varphi_0](y)
      &=\rho_{00}((b^*y,0))\varphi_0(y), &                                       b&\in\Hom(V_{n_0},X), \\
    [\omega_0(u^\mathrm{c}(c))\varphi_0](y)
      &= \psi(\tfrac{1}{2}\langle u^\mathrm{c}(c)y,y\rangle)\varphi_0(y), &               c&\in\Alt(X^*,X), \\
    [\omega_0(w_X)\varphi_0](y)
      &=\omega_{00}((-1_{V_{n_0}})^k) \int_{X^*\otimes W_{n_0}}\psi(-\langle y',I_X y\rangle)\varphi_0(y')dy'.
\end{align*}

For $\varphi_{00}\in\calS_{00}=\calS(V_{n_0}\otimes W_{02})$ and $x\in V_{n_0}\otimes W_{02}$,
\begin{align*}
    [\omega_{00}((-1_{W_{n_0}})^k)\varphi_{00}] (x)
      &=\gamma_F((-1)^{k(n-k)},\psi)^{-1}\varphi_{00}((-1)^kx), \\
    [\omega_{00}((-1_{V_{n_0}})^k)\varphi_{00}](x)
      &=\varphi_{00}((-1)^kx).
\end{align*}
%%%%%%%%%%%%%%%%%%%%%%%%%%%%%%%%%%%%%%%%%%%%%%%%%%%%%%%%%%%%%%%%%%%%%%%%%%%%%%%%%%%%%%%%%
\subsection{Gan-Ichino's equivariant maps}\label{eqmap}
Next, we shall construct equivariant maps which realize the theta correspondence.
Put
\begin{align*}
f_\calS(\varphi)(gh)
&=\left[[\omega(gh)\varphi](e)\right](0), \\
\hat{f}_\calS(\varphi)(gh)
&=\left[\int_{X\otimes Y^*}[\omega(gh)\varphi](x)\psi(-\langle e^*,x\rangle)dx\right](0),
\end{align*}
for $\varphi \in \calS =\calS(V\otimes Y^*) \otimes \calS(X^*\otimes W_{n_0}) \otimes \calS_{00}$, $g \in \Mp(W)$, and $h \in \Or(V)$.
If $f=f_\calS(\varphi)$ or $\hat{f}_\calS(\varphi)$, then by the explicit formulas of the mixed Schr\"odinger model, we have
\begin{align}\label{formulaf}
f(nugh)&=f(gh),& &n\in N, u\in U, \notag\\
f(g_0h_0gh)&=\omega_{00}(g_0h_0)f(gh),& &g_0 \in \Sp(W_{n_0}), h_0 \in \Or(V), \notag\\
f(m(a)l(a)gh)&=\gf(\det a,\psi)\inv|\det a|_F^{\rho_Q+\rho_P}f(gh),& &a\in \GL_k(F)\cong \GL(X)\cong \GL(Y),
\end{align}
for any $g\in \Mp(W),$ and $h\in \Or(V)$.
In the rest of this section, we shall drop the subscript $\calS$ for simplicity.

In this subsection, we shall write $\tau=\tau_1$ and assume that $\sigma_0$ and $\pi_0$ may be direct sums of irreducible tempered representations, whose summands have a same $L$-parameter $\phi_0$ and correspond bijectively via $\Theta_\psi$.
For $\rho=\tau, \pi_0, \sigma_0$, let $(\rho^\vee, \scrV_{\rho^\vee})$ be the contragredient representation of $(\rho, \scrV_\rho)$, and $\an{-,-}$ the invariant non-degenerate bilinear form on $\scrV_\rho \times \scrV_{\rho^\vee}$.
We fix a non-zero $\Mp(W_{n_0}) \times \SO(V_{n_0})$-equivariant map
\begin{equation}\label{00}
\calT_{00} : \omega_{00} \otimes \sigma_0^\vee \lra \pi_0.
\end{equation}
For any $\varphi \in \calS$, $\scrF_s\in \Ind_Q^{\SO(V)}(\tau_s \otimes \sigma_0^\vee)$, $g\in \Mp(W)$, $\cv \in \scrV_{\tau^\vee}$, and $\cv_0 \in \scrV_{\pi_0^\vee}$, put
\begin{equation*}
I(s,\varphi\otimes \scrF_s, \cv\otimes\cv_0, g)
=\frac{1}{L(s+\tfrac{1}{2},\tau)}
         \int_{U\SO(V_{n_0})\backslash \SO(V)}
          \an{ \calT_{00}(\hat{f}(\varphi)(gh)\otimes \an{ \scrF_s(h), \cv }), \cv_0 }
         dh,
\end{equation*}
if the right hand side converges absolutely.
Here, $s \in \C$ is a complex variable.
\begin{lem}\label{deft}
We have the following.
\begin{enumerate}[(1)]
  \item The integral $I(s,\varphi\otimes \scrF_s, \cv\otimes\cv_0, g)$ converges absolutely for $\myre(s)>-\frac{1}{2}$, and admits a holomorphic continuation to $s \in \C$.
  \item For $\myre(s)<\frac{1}{2}$, we have that $I(s,\varphi\otimes \scrF_s, \cv\otimes\cv_0, g)$ is equal to
           \begin{equation*}
            L(s+\frac{1}{2},\tau)^{-1} \gamma(s+\frac{1}{2},\tau,\psi)^{-1}
             \int_{U\SO(V_{n_0})\backslash \SO(V)}
                 \an{ \calT_{00}(f(\varphi)(gh)\otimes \an{ \scrF_s(h), \cv }), \cv_0 } dh.
           \end{equation*}
   \item By virtue of (1), we define a vector $\calT_s( \varphi \otimes \scrF_s)(g)$ of $\scrV_\tau \otimes \scrV_{\pi_0}$ by
           \begin{equation*}
            \an{ \calT_s( \varphi \otimes \scrF_s)(g), \cv \otimes \cv_0 }
             =I(s, \varphi \otimes \scrF_s, \cv \otimes \cv_0, g).
           \end{equation*}
           Then for any $0 \neq \scrF \in \Ind_Q^{\SO(V)}(\tau \otimes \sigma_0^\vee)$, there exists $\varphi \in \calS$ such that
            \begin{equation*}
             \calT_{s=0}(\varphi \otimes \scrF) \neq0.
            \end{equation*}
\end{enumerate}
\end{lem}
\begin{proof}
The proof is similar to those of Lemmas 8.1, 8.2, and 8.3 in \cite{gigp}.
\end{proof}
Now, when $s=0$ the assignment $\varphi \otimes \scrF \mapsto \calT_0( \varphi \otimes \scrF)$ gives an $\Mp(W) \times \SO(V)$-equivariant map $\omega \otimes \Ind_Q^{\SO(V)}(\tau \otimes \sigma_0^\vee) \to \Ind_P^{\Mp(W)}(\wtil{\tau}\otimes\pi_0)$.
We shall write $\calT(k, \calT_{00})$ for this map.\\

Now, we note the functorialities of the equivariant map $\calT(k,\calT_{00})$ here.
We have the following two lemmas, which easily follow from the definition of $\calT(k,\calT_{00})$.
\begin{lem}\label{adjacent}
Let $(\tau', \scrV_{\tau'})$ be a representation of $\GL_k$ that is isomorphic to $\tau$, and $A : (\tau, \scrV_\tau) \to (\tau', \scrV_{\tau'})$ an isomorphism of representations of $\GL_k$.
Then the diagram
\begin{align*}
\begin{CD}
\omega\otimes\Ind_Q^{\SO(V)}(\tau\otimes\sigma_0^\vee) @>\calT(k,\calT_{00})>> \Ind_P^{\Mp(W)}(\wtil{\tau}\otimes\pi_0)\\
 @V1\otimes \Ind(A)VV                                                      @VV\Ind(A)V\\
 \omega\otimes\Ind_Q^{\SO(V)}(\tau'\otimes\sigma_0^\vee) @>>\calT(k,\calT_{00})> \Ind_P^{\Mp(W)}(\wtil{\tau'}\otimes\pi_0)
\end{CD}
\end{align*}
commutes.
Here, $\Ind(A)$ denotes an operator defined by $[\Ind(A)\scrF](x) = A(\scrF(x))$.
\end{lem}

\begin{lem}\label{reflection}
Let $(\sigma_0', \scrV_{\sigma_0'})$ (resp. $(\pi_0', \scrV_{\pi_0'})$) be a representation of $\SO(V_{n_0})$ (resp. $\Mp(W_{n_0})$) that is isomorphic to $\sigma_0$ (resp. $\pi_0$), and $B : (\sigma_0^\vee, \scrV_{\sigma_0^\vee}) \to ({\sigma_0'}^\vee, \scrV_{{\sigma_0'}^\vee})$ (resp. $C : (\pi_0, \scrV_{\pi_0}) \to (\pi_0', \scrV_{\pi_0'}$)) an isomorphism.
Choose an $\Mp(W_{n_0}) \times \SO(V_{n_0})$-equivariant map $\calT_{00}' : \omega_{00} \otimes {\sigma_0'}^\vee \to \pi_0$ such that the diagram
\begin{align*}
\begin{CD}
\omega_{00}\otimes\sigma_0^\vee @>\calT_{00}>> \pi_0 \\
 @V1\otimes BVV                                  @VVCV \\
\omega_{00}\otimes{\sigma_0'}^\vee @>>\calT_{00}'> \pi_0'
\end{CD}
\end{align*}
commutes.
Then the diagram
\begin{align*}
\begin{CD}
\omega\otimes\Ind_Q^{\SO(V)}(\tau\otimes\sigma_0^\vee) @>\calT(k,\calT_{00})>> \Ind_P^{\Mp(W)}(\wtil{\tau}\otimes\pi_0)\\
 @V1\otimes \Ind(B)VV                                                      @VV\Ind(C)V\\
 \omega\otimes\Ind_Q^{\SO(V)}(\tau\otimes{\sigma_0'}^\vee) @>>\calT(k,\calT_{00}')> \Ind_P^{\Mp(W)}(\wtil{\tau}\otimes\pi_0')
\end{CD}
\end{align*}
also commutes.\\
\end{lem}

Finally, we remark a key property of the assignment $\calT_s$.
\begin{prop}\label{comm}
For $\varphi \in \calS$ and $\scrF_s \in \Ind_Q^{\SO(V)}(\tau_s \otimes \sigma_0^\vee)$, we have
\begin{equation*}
\calR_P(w_M, \wtil{\tau}_s \otimes \pi_0) \calT_s( \varphi \otimes \scrF_s) 
= \beta(s) \cdot \calT_{-s}( \varphi \otimes \calR_Q(w_L, \tau_s \otimes \sigma_0^\vee)\scrF_s),
\end{equation*}
where
\begin{equation*}
\beta(s)
=|2|_F^{2ks} 
  \cdot \frac{L(-s+\tfrac{1}{2},\tau^\vee)}{L(s+\tfrac{1}{2}, \tau)}
    \cdot  \frac{\gamma(-s+\tfrac{1}{2}, \tau^\vee,\psi)}{\gamma(s+\tfrac{1}{2},\tau, \psi)}.
\end{equation*}
\end{prop}
\begin{proof}
Noting that $\phi_0^\vee \cong \phi_0$ and $\gf(\psi \circ q_V) = \epsilon(V) \gf(\psi)$, one can prove it by a similar argument to the proof of \cite[Corollary 8.5]{gigp}.
\end{proof}
%%%%%%%%%%%%%%%%%%%%%%%%%%%%%%%%%%%%%%%%%%%%%%%%%%%%%%%%%%%%%%%%%%%%%%%%%%%%%%%%%%%%%%%%%%%%%%
%%%%%%%%%%%%%%%%%%%%%%%%%%%%%%%%%%%%%%%%%%%%%%%%%%%%%%%%%%%%%%%%%%%%%%%%%%%%%%%%%%%%%%%%%%%%%%
\section{Proof of Proposition \ref{reduction}}
Now we can define an equivariant map $\calT$ desired in Proposition \ref{reduction}, and give our proof of the proposition.
We will define such a map $\calT$ to be the map $\calT(k, \calT_{00})$ constructed in \S\ref{eqmap} when $P$ is maximal, and by induction in stages when $P$ is not maximal.
We shall use the same notation as in \S \ref{Io}, and assume that $\sigma_0=\Theta_\psi(\pi_0)$.
\subsection{An equivariant map $\calT$}
In this subsection, we shall define an $\Mp(W) \times \SO(V)$-equivariant map
\begin{equation*}
\calT : \omega_{V,W,\psi} \otimes \Ind_Q^{\SO(V)}(\breve{\sigma}_L) \lra \Ind_P^{\Mp(W)}(\pi_M)
\end{equation*}
that will satisfy Proposition \ref{reduction}.
For a fixed $1\leq m' \leq m$, we put $\bfk'=(k_1, \ldots, k_{m'})$, $k'=k_1+\cdots+k_{m'}$, $\bfk''=(k_{m'+1}, \ldots, k_m)$, $k''=k_{m'+1}+\cdots+k_m$, and $n'=n-k'$.
As in \S\ref{maxparab} and \S\ref{mixmodel}, we shall take $X=X_k$, $X^*=X^*_k$, $Y= Y_k$, and $Y^*= Y^*_k$.
Put $W_{02}= \myspan_F(y^*_{k+1},\dots ,y^*_n)$.
Also, let us put
\begin{align*}
 X'   &=X_{k'}=\myspan_F(x_1,\dots,x_{k'}),&    {X'}^*&=X_{k'}^*=\myspan_F(x^*_1,\dots, x^*_{k'}), \\
 Y'&= Y_{k'} = \myspan_F(y_1,\dots ,y_{k'}),&   {Y'}^*&= Y_{k'}^* = \myspan_F(y^*_1,\dots ,y_{k'}^*), \\
 X''&=\myspan_F(x_{k'+1},\ldots,x_k),&  {X''}^*&= \myspan_F(x^*_{k'+1},\ldots,x^*_k),\\
 Y''&=\myspan_F(y_{k'+1},\ldots,y_k),&  {Y''}^*&= \myspan_F(y^*_{k'+1},\ldots,y^*_k),
\end{align*}
and $V' = V_{n'}$, $W'=W_{n'}$, so that
\begin{align*}
 &V=X' \oplus V' \oplus {X'}^*,& &V'= X'' \oplus V_{n_0} \oplus {X''}^*,&\\
 &W=Y' \oplus W' \oplus {Y'}^*,& &W'= Y'' \oplus W_{n_0} \oplus {Y''}^*,&
\end{align*}
and we shall write $Q' = L' \ltimes U'$ and $P' = M' \ltimes N'$ for the maximal parabolic subgroups of $\SO(V')$ and $\Mp(W')$ stabilizing $X''$ and $Y''$, respectively.\\

Let $(\omega, \calS)$, $(\omega_0, \calS_0)$, and $(\omega_{00}, \calS_{00})$ be the models of the Weil representations constructed in \S\ref{mixmodel}.
Additionally, let $\omega''$ be the realization of the Weil representation $\omega_{V',W',\psi}$ of $\Or(V') \times \Mp(W')$ on a mixed model
\begin{equation*}
\calS''
= \calS(V' \otimes {Y''}^*) \otimes \calS({X''}^* \otimes W_{n_0}) \otimes \calS_{00},
\end{equation*}
and let $\omega' = \omega_{V,W,\psi}$ be the realization of the Weil representation of $\Or(V)\times\Mp(W)$ on a mixed model
\begin{equation*}
\calS'
= \calS(V\otimes {Y'}^*) \otimes \calS({X'}^* \otimes W') \otimes \calS''.
\end{equation*}
As in \S\ref{mixmodel}, fix isomorphisms
\begin{equation}\label{isommodel}
(\omega, \calS) \cong (\omega^{\mathrm{or}}, \calS^{\mathrm{or}}) \cong (\omega', \calS')
\end{equation}
of the three realizations of $\omega_{V,W,\psi}$, and identify them.\\

Let $P^{\GL}_\bfk$ be the standard parabolic subgroup of $\GL_k \cong \GL(Y)$ stabilizing flag
\begin{equation*}
Y_{k_1} \subset Y_{k_1+k_2} \subset \dots \subset Y_{k_1+\cdots +k_m},
\end{equation*}
Similarly, we define the standard parabolic subgroups $P^{\GL}_{\bfk'}$ of $\GL_{k'}$ and $P^{\GL}_{\bfk''}$ of $\GL_{k''}$.
Put $\tau=\Ind_{P^{\GL}_\bfk}^{\GL_{k}}(\tau_1 \otimes \cdots \otimes \tau_m)$, $\tau' = \Ind_{P^{\GL}_{\bfk'}}^{\GL_{k'}}(\tau_1 \otimes \cdots \otimes \tau_{m'})$, and $\tau''=\Ind_{P^{\GL}_{\bfk''}}^{\GL_{k''}}(\tau_{m'+1}\otimes\cdots\otimes\tau_m)$.
These representations are irreducible, since $\tau_1, \ldots, \tau_m$ are tempered.
Define canonical isomorphisms
\begin{align*}
\Phi &: \Ind_Q^{\SO(V)}(\tau_1 \otimes \cdots \otimes \tau_m \otimes \sigma_0^\vee)
             \lra
              \Ind_{Q_k}^{\SO(V)}( \tau \otimes \sigma_0^\vee ),\\
\Psi &: \Ind_Q^{\SO(V)}(\tau_1 \otimes \cdots \otimes \tau_m \otimes \sigma_0^\vee)
            \lra
             \Ind_{Q_{k'}}^{\SO(V)} \left( \tau' \otimes \Ind_{Q'}^{\SO(V')}(\tau'' \otimes \sigma_0^\vee)  \right),
\end{align*}
by
\begin{align*}
\Phi \scrF (h)(x)
&= \delta_{Q_k}(l(x))^{-\frac{1}{2}} \scrF(l(x)h),\\
\Psi \scrF (h)(x',h')
&= \delta_{Q_{k'}}(l'(x'))^{-\frac{1}{2}} \scrF(l'(x')h'h),
\end{align*}
where $l$ and $l'$ are the canonical embeddings $\GL_k \inj L_k$ and $\GL_{k'} \inj L_{k'}$, as in \S\ref{maxparab}, respectively.

Similarly, by abuse of notation, we shall take canonical isomorphisms
\begin{align*}
\Phi &: \Ind_P^{\Mp(W)}(\wtil{\tau_1} \otimes \cdots \otimes \wtil{\tau_m} \otimes \pi_0)
          \lra
           \Ind_{P_k}^{\Mp(W)}( \wtil{\tau} \otimes \pi_0 ),\\
\Psi &: \Ind_P^{\Mp(W)}(\wtil{\tau_1} \otimes \cdots \otimes \wtil{\tau_m} \otimes \pi_0)
         \lra
           \Ind_{P_{k'}}^{\Mp(W)} \left( \wtil{\tau'} \otimes \Ind_{P'}^{\Mp(W')}(\wtil{\tau''} \otimes \pi_0) \right),
\end{align*}
and the canonical embeddings $m : \GL_{k} \inj \mpsp{M_k}$ and $m' : \GL_{k'} \inj \mpsp{M_{k'}}$.

Next, following \S \ref{eqmap}, we put $\calT^a = \calT(k,\calT_{00})$ and $\calT^r = \calT(k',\calT(k'',\calT_{00}))$, which are $\Mp(W) \times \SO(V)$-equivariant maps
\begin{equation*}
\omega \otimes \Ind_Q^{\SO(V)}\left(\tau \otimes \sigma_0^\vee\right)
\lra
\Ind_P^{\Mp(W)}\left(\wtil{\tau} \otimes \pi_0\right),
\end{equation*}
and
\begin{equation*}
\omega' \otimes \Ind_{Q_{k'}}^{\SO(V)}\left(\tau' \otimes \Ind_{Q'}^{\SO(V')}(\tau'' \otimes \sigma_0^\vee)\right)
\lra
\Ind_{P_{k'}}^{\Mp(W)}\left(\wtil{\tau'} \otimes \Ind_{P'}^{\Mp(W')}(\wtil{\tau''} \otimes \pi_0)\right),
\end{equation*}
respectively.
Here $\calT_{00}$ is the fixed map \eqref{00}.
\begin{lem}\label{tc}
The diagram
\begin{align*}
\begin{CD}
 \omega \otimes \Ind_{Q_k}^{\SO(V)}\left(\tau \otimes \sigma_0^\vee\right)
 @>\calT^a>> \Ind_{P_k}^{\Mp(W)}\left(\wtil{\tau} \otimes \pi_0\right) \\
 @A1\otimes\Phi AA  @AA\Phi A \\
 \omega_{V,W,\psi} \otimes \Ind_Q^{\SO(V)}(\tau_1 \otimes \cdots \otimes \tau_m \otimes \sigma_0^\vee)
 @. \Ind_P^{\Mp(W)}(\wtil{\tau_1} \otimes \cdots \otimes \wtil{\tau_m} \otimes \pi_0) \\
 @V1 \otimes \Psi VV  @VV\Psi V\\
 \omega' \otimes \Ind_{Q_{k'}}^{\SO(V)}\left(\tau' \otimes \Ind_{Q'}^{\SO(V')}(\tau'' \otimes \sigma_0^\vee)\right)
 @>> \calT^r>
 \Ind_{P_{k'}}^{\Mp(W)}\left(\wtil{\tau'} \otimes \Ind_{P'}^{\Mp(W')}(\wtil{\tau''} \otimes \pi_0)\right)
\end{CD}
\end{align*}
commutes.
\end{lem}

Lemma \ref{tc} lets us define an $\Mp(W)\times\SO(V)$-equivariant map
\begin{equation*}
\calT : \omega_{V, W, \psi} \otimes \Ind_Q^{\SO(V)}(\breve{\sigma}_L) \lra \Ind_P^{\Mp(W)}({\pi}_M),
\end{equation*}
so that the diagram will remain commutative if we insert $\calT$ into the middle horizontal space.
In other words,
\begin{equation*}
\calT = \Phi\inv \circ \calT^a \circ (1\otimes\Phi)
= \Psi\inv \circ \calT^r \circ (1 \otimes \Psi).
\end{equation*}
%%%%%%%%%%%%%%%%%%%%%%%%%%%%%%%%%%%%%%%%%%%%%%%%%%%%%%%%%%%%%%%%%%%%%%%%%%%%%%%%%%%%%%%%%
\subsection{Proof of Lemma \ref{tc}}
Let $\varphi \in \calS \cong \calS'$ and $\scrF \in \Ind_Q^{\SO(V)}(\tau_1 \otimes \cdots \otimes \tau_m \otimes \sigma_0^\vee)$.
It suffices to show that
\begin{equation}\label{ppc}
 \an{\calT^r (\varphi \otimes \Psi\scrF)(g)(1,1), \cv_1 \otimes \cdots \otimes \cv_m \otimes \cv_0}
 = \an{\calT^a (\varphi \otimes \Phi\scrF)(g)(1), \cv_1 \otimes \cdots \otimes \cv_m \otimes \cv_0},
\end{equation}
for any $\cv_i \in \scrV_{\tau_i^\vee}$, $\cv_0 \in \scrV_{\pi_0^\vee}$, and $g \in \Mp(W)$. \\

Fix $\cv_i \in \scrV_{\tau_i^\vee}$ $\cv_0 \in \scrV_{\pi_0^\vee}$, and $g \in \Mp(W)$.
Choose an element $\scrK = \calK \otimes \scrK'$ of
\begin{equation*} 
\Ind_{P^{\GL}_{\bfk'}}^{\GL_{k'}}(\tau_1^\vee \otimes \cdots \otimes \tau_{m'}^\vee)
 \otimes \Ind_{P'}^{\Mp(W')}(\wtil{\tau''}^\vee \otimes \pi_0^\vee)
\cong {\tau'}^\vee \otimes \Ind_{P'}^{\Mp(W')}(\wtil{\tau''} \otimes \pi_0)^\vee,
\end{equation*}
such that
\begin{align*}
\supp(\scrK) \subset (P^{\GL}_{\bfk'} \times P') \cdot K',\\
\scrK(x) = \cv_1 \otimes \cdots \otimes \cv_m \otimes \cv_0,
\end{align*}
for any $x \in K'$, where $K' = K'_{\mathit G} \times K'_{\mathit M} \subset \GL_{k'} \times \Mp(W')$ is a compact open subgroup such that
\begin{itemize}
\item
 $\left((\GL_{k_1} \times \cdots \times \GL_{k_{m''}}) \times \Mp(W')\right) \cap K'$ stabilizes $\cv_1, \ldots, \cv_m,\cv_0$;
\item
 $K'$ stabilizes $\omega'(g)\varphi$, i.e., $\omega'(m'(a')g'_0g)\varphi = \omega'(g)\varphi$ for any $(a',g'_0) \in K'$.
\end{itemize}
Since $K'$ stabilizes $\calT^r (\varphi \otimes \Psi\scrF)(g) = \calT^r(\omega'(g)\varphi \otimes \Psi\scrF)(1)$, we have
\begin{align}\label{tj1}
\an{\calT^r (\varphi \otimes \Psi\scrF)(g), \scrK}
&=\int_{(P^{\GL}_{\bfk'} \times P_{k''})\backslash (\GL_{k'}\times\Mp(W'))}\an{\calT^r (\varphi \otimes \Psi\scrF)(g)(x), \scrK(x)} dx\notag\\
&=\int_{(P^{\GL}_{\bfk'} \times P_{k''})\backslash (P^{\GL}_{\bfk'} \times P_{k''}) K'}\an{\calT^r (\varphi \otimes \Psi\scrF)(g)(x), \scrK(x)} dx\notag\\
&=\vol(K')\an{\calT^r (\varphi \otimes \Psi\scrF)(g)(1,1),\cv_1 \otimes \cdots \otimes \cv_m \otimes \cv_0}.
\end{align}

On the other hand, by the definition of $\calT^r$ and Lemma \ref{deft}, we see that $\an{\calT^r (\varphi \otimes \Psi\scrF)(g), \scrK}$ equals
\begin{equation*}
L(\tfrac{1}{2}, \tau')\inv \gamma(\tfrac{1}{2}, \tau', \psi)\inv
    \int_{U_{k'}\SO(V_{n'})\backslash\SO(V)}
     \an{\calT'(f_{\calS'}(\varphi)(gh) \otimes \an{\Psi\scrF(h),\calK}), \scrK'}
    dh,
\end{equation*}
where
\begin{equation*}
\calT' = \calT(k'',\calT_{00})
 : \omega'' \otimes \Ind_{Q'}^{\SO(V')}(\tau'' \otimes \sigma_0^\vee)
  \lra \Ind_{P'}^{\Mp(W')}(\wtil{\tau''} \otimes \pi_0).
\end{equation*}
The last integral is equal to
\begin{align*}
&\int_{U_{k'}\SO(V_{n'})\backslash\SO(V)}
       \An{\calT'\left(f_{\calS'}(\varphi)(gh) \otimes 
        \int_{P^{\GL}_{\bfk'}\backslash\GL_{k'}}
         \an{\Psi\scrF(h)(a',\bullet),\calK(a')} da'\right) , \scrK'} dh\\
&=\int_{U_{k'}\SO(V_{n'})\backslash\SO(V)}
       \An{\calT'\left(f_{\calS'}(\varphi)(gh) \otimes 
        \int_{K'_{\mathit G}}
         \an{\Psi\scrF(l'(a')h)(1,\bullet), \cv_1\otimes\cdots\otimes\cv_{m'}} da'\right) , \scrK'} dh\\
&=\int_{K'_{\mathit G}}
       \int_{U_{k'}\SO(V_{n'})\backslash\SO(V)}
       \An{\calT'\left(f_{\calS'}(\varphi)(gl'(a')h) \otimes
         \an{\Psi\scrF(h)(1,\bullet), \cv_1\otimes\cdots\otimes\cv_{m'}} \right) , \scrK'} dhda'.
\end{align*}
Thus we have that $\an{\calT^r (\varphi \otimes \Psi\scrF)(g), \scrK}$ is equal to the product of $L(\tfrac{1}{2}, \tau')\inv\gamma(\tfrac{1}{2}, \tau', \psi)\inv$ and
\begin{align}\label{stst}
 \int_{K'_{\mathit G}} \int_{U_{k'}\SO(V_{n'})\backslash\SO(V)}
       \An{\calT'\left(f_{\calS'}(\varphi)(gl'(a')h) \otimes
         \an{\Psi\scrF(h)(1,\bullet), \cv_1\otimes\cdots\otimes\cv_{m'}} \right) , \scrK'} dhda'.
\end{align}
Moreover,
\begin{align}\label{tj3}
&\int_{K'_{\mathit G}}
       \An{\calT'\left(f_{\calS'}(\varphi)(gl'(a')h) \otimes
         \an{\Psi\scrF(h)(1,\bullet), \cv_1\otimes\cdots\otimes\cv_{m'}} \right) , \scrK'} da'\notag\\
&=\int_{K'_{\mathit G}}
      \int_{K'_{\mathit M}}
      \An{\calT'\left(
        f_{\calS'}(\varphi)(gl'(a')h) \otimes
         \an{\Psi\scrF(h)(1,\bullet), \cv_1\otimes\cdots\otimes\cv_{m'}}
          \right) (g'_0) , \scrK'(g'_0)} dg'_0 da'\notag\\
&=\int_{K'}
       \An{\calT'\left(
        f_{\calS'}(\varphi)(g'_0gl'(a')h) \otimes
         \an{\Psi\scrF(h)(1,\bullet), \cv_1\otimes\cdots\otimes\cv_{m'}}
          \right) (1) , \cv_{m'+1}\otimes\cdots\otimes\cv_0} d(a',g'_0).
\end{align}
Then \eqref{stst} and \eqref{tj3} imply that $\an{\calT^r (\varphi \otimes \Psi\scrF)(g), \scrK}$ is
\begin{align}\label{ststst}
&L(\tfrac{1}{2}, \tau')\inv\gamma(\tfrac{1}{2}, \tau', \psi)\inv
       \int_{U_{k'}\SO(V_{n'})\backslash\SO(V)}
       \int_{K'} \notag\\
&\times
       \An{\calT'\left(
        f_{\calS'}(\varphi)(g'_0gl'(a')h) \otimes
         \an{\Psi\scrF(h)(1,\bullet), \cv_1\otimes\cdots\otimes\cv_{m'}}
          \right) (1) , \cv_{m'+1}\otimes\cdots\otimes\cv_0} d(a',g'_0)dh.
\end{align}
Now by the formula \eqref{formulaf} and the choice of $K'$, we have
\begin{equation*}
f_{\calS'}(\varphi)(g'_0gl'(a')uh)
=f_{\calS'}(\varphi)(guh).
\end{equation*}
Therefore, \eqref{tj1} and \eqref{ststst} imply that
\begin{align}\label{tj4}
&\an{\calT^r (\varphi \otimes \Psi\scrF)(g)(1,1), \cv_1 \otimes \cdots \otimes \cv_m \otimes \cv_0}\notag\\
&=L(\tfrac{1}{2}, \tau')\inv \gamma(\tfrac{1}{2}, \tau', \psi)\inv
     \int_{U_{k'}\SO(V_{n'})\backslash\SO(V)}\notag\\
 &\quad\times
       \An{\calT'\left(
        f_{\calS'}(\varphi)(gh) \otimes
         \an{\Psi\scrF(h)(1,\bullet), \cv_1\otimes\cdots\otimes\cv_{m'}}
          \right) (1) , \cv_{m'+1}\otimes\cdots\otimes\cv_0} dh.
\end{align}
Now, the definition of $\calT'$ gives that the last integral is equal to
\begin{align}\label{tj5}
&L(\tfrac{1}{2}, \tau'')\inv \gamma(\tfrac{1}{2}, \tau'',\psi)\inv
       \int_{U_{k'}\SO(V_{n'})\backslash\SO(V)}
       \int_{U'\SO(V_{n_0}) \backslash \SO(V')}\notag\\
 &\times
        \An{\calT_{00}
         \left(
          f_{\calS''}
           \left(
            f_{\calS'}(\varphi)(gh)
           \right)
           (h')
           \otimes
          \an{\an{\Psi\scrF(h)(1,h'), \cv_1\otimes\cdots\otimes\cv_{m'}}, \cv_{m'+1}\otimes\cdots\otimes\cv_m}
         \right), \cv_0} dh'dh\notag\\
=&L(\tfrac{1}{2}, \tau'')\inv \gamma(\tfrac{1}{2}, \tau'',\psi)\inv
       \int_{U\SO(V_{n_0}) \backslash\SO(V)}
       \int_{U_{k'}U'\backslash U}\notag\\
 &\times
       \An{\calT_{00}\left(
       f_{\calS''}
         \left(
         f_{\calS'}(\varphi)(guh)
         \right)
         (1)
          \otimes
         \an{\scrF(h), \cv_1\otimes\cdots\otimes\cv_m}
          \right), \cv_0} dudh.
\end{align}

By \eqref{tj4} and \eqref{tj5}, we have
\begin{align}\label{ppcl}
&\an{\calT^r (\varphi \otimes \Psi\scrF)(g)(1,1), \cv_1 \otimes \cdots \otimes \cv_m \otimes \cv_0}\notag\\
&=L(\tfrac{1}{2}, \tau)\inv \gamma(\tfrac{1}{2}, \tau,\psi)\inv
       \int_{U\SO(V_{n_0}) \backslash\SO(V)}
       \An{\calT_{00}\left(
          f''(\varphi)(gh)
           \otimes
          \scrF_0(h)
           \right), \cv_0
       } dh,
\end{align}
where
\begin{align*}
f''(\varphi) (gh) &= \int_{U_{k'}U'\backslash U}
         f_{\calS''}
          \left(
          f_{\calS'}(\varphi)(ugh)
          \right)
          (1)
          du,\\
\scrF_0 (h) &= \an{\scrF(h), \cv_1\otimes\cdots\otimes\cv_m}.\\
\end{align*}

Similarly, we can obtain
\begin{align}\label{ppcr}
&\an{\calT^a (\varphi \otimes \Phi\scrF)(g)(1), \cv_1 \otimes \cdots \otimes \cv_m \otimes \cv_0}\notag\\
&= L(\tfrac{1}{2}, \tau)\inv \gamma(\tfrac{1}{2}, \tau,\psi)\inv
   \int_{U\SO(V_{n_0})\backslash\SO(V)}\An{\calT_{00}(f'(\varphi)(gh) \otimes \scrF_0(h)), \cv_0}dh,
\end{align}
where
\begin{equation*}
f'(\varphi) (gh)
=\int_{U_k\backslash U}f_\calS(\varphi)(ugh) du.
\end{equation*}

Now \eqref{ppcl} and \eqref{ppcr} say that it suffices to show that $f''(\varphi)=f'(\varphi)$, which will follow from Lemma \ref{toraja} below.
\qed

\begin{lem}\label{toraja}
Under the identification \eqref{isommodel}, for any $\varphi \in \calS^{\mathrm{or}}$, we have
\begin{equation*}
\int_{U_{k'}U'\backslash U}  f_{\calS''} \left( f_{\calS'}(\varphi)(u) \right) (1)du
=  \int_{U_k\backslash U} f_\calS(\varphi)(u) du.
\end{equation*}
\end{lem}
\begin{proof}
Put
\begin{align*}
 e' &=x_1\otimes y^*_1+\dots+x_{k'}\otimes y_{k'}^* \in X'\otimes {Y'}^*, \\
 e''&=x_{k'+1}\otimes y^*_{k'+1}+\dots+x_k\otimes y_k^* \in X''\otimes {Y''}^*,
\end{align*}
and let $\varphi \in \calS^{\mathrm{or}}$.
Because
\begin{align*}
\calS^{\mathrm{or}}
 &\cong \calS(V\otimes {Y'}^*) \ 
    \otimes \ \calS((X'\oplus {X'}^*)\otimes ({Y''}^*\oplus W_{02})) \\
 &\quad\otimes \ \calS(V'\otimes {Y''}^*) \ 
    \otimes \ \calS((X''\oplus {X''}^*)\otimes W_{02}) \ 
     \otimes \ \calS_{00},
\end{align*}
we shall write
\begin{align*}
\varphi
 \left[
  x,
  \left(\begin{array}{c} y_1\\ y_2 \end{array} \right), 
  x', 
  \left(\begin{array}{c} y'_1 \\ y'_2 \end{array}\right)
 \right]
=\varphi
 \left[
  x,
  \left(\begin{array}{c} y_1\\ y_2 \end{array}\right)
 \right]
 \left[
  x', 
  \left(\begin{array}{c} y'_1 \\ y'_2 \end{array}\right)
 \right]
\end{align*}
for the evaluation of $\varphi$ at $x \in V \otimes {Y'}^*$, $y_1 \in X' \otimes ({Y''}^*\oplus W_{02})$, $y_2 \in {X'}^*\otimes ({Y''}^*\oplus W_{02})$, $x' \in V'\otimes {Y''}^*$, $y'_1 \in X''\otimes W_{02}$, and $y'_2 \in {X''}^*\otimes W_{02}$, which is an element of $\calS_{00}$.
Then we have
\begin{align*}
f_{\calS''} \left( f_{\calS'}(\varphi)(u) \right) (1)
&=\int_{y' \in X''\otimes W_{02}}
     \left[f_{\calS'}(\varphi)(u)\right] \left[e'', \left(\begin{array}{c}y'\\0\end{array}\right)\right] dy'\\
&=\int_{y' \in X''\otimes W_{02}} \int_{y' \in X' \otimes ({Y''}^*\oplus W_{02})}
      \omega^{\mathrm{or}}(u)\varphi
       \left[
        e', \left(\begin{array}{c}y\\0\end{array}\right), e'', \left(\begin{array}{c}y'\\0\end{array}\right)
       \right] dy'\\
&=\int_{y' \in X''\otimes W_{02}} \int_{y' \in X' \otimes ({Y''}^*\oplus W_{02})}
      \varphi
       \left[
        u\inv e', \left(\begin{array}{c}y\\0\end{array}\right), e'', \left(\begin{array}{c}y'\\0\end{array}\right)
       \right] dy',
\end{align*}
for any $u \in U_{k'}U'\backslash U$.
Thus, if we regard $\varphi$ as an element of
\begin{equation*}
\calS^\mathrm{or}
\cong \calS((V\otimes Y^*) \oplus ((X\oplus X^*)\otimes W_{02})) \otimes \calS_{00},
\end{equation*}
then we have
\begin{equation}\label{lht}
\int_{U_{k'}U'\backslash U}  f_{\calS''} \left( f_{\calS'}(\varphi)(u) \right) (1)du
= \int_{c=(c_{i,j}) \in C} \varphi \left( \sum_{i=1}^k x_i\otimes y_i^* +  \sum_{i,j} c_{i,j} x_i \otimes y_j^* \right) \prod_{i,j}d_\psi c_{i,j},
\end{equation}
where the integration region $C$ is a direct product $C=C_1 \times \cdots \times C_m$ of sets
\begin{align*}
C_l
=\Set{
(c_{i,j})
|
c_{i,j}\in F,
\begin{array}{l}
  i=k_0+\cdots+k_{l-1}+1, \ldots, k_0+\cdots+k_l, \\
  j=k_0+\cdots+k_l+1, \ldots, n,
\end{array}
}
\end{align*}
of $k_l$ $\times$ $(k_{l+1}+\cdots+k_m)$ matrices with certain shifted indices.
Similarly, we have
\begin{equation}\label{rht}
\int_{U_k\backslash U} f_\calS(\varphi)(u) du
= \int_{c=(c_{i,j}) \in C} \varphi \left( \sum_{i=1}^k x_i\otimes y_i^* +  \sum_{i,j} c_{i,j} x_i \otimes y_j^* \right) \prod_{i,j}d_\psi c_{i,j}.
\end{equation}
Now the lemma follows from \eqref{lht} and \eqref{rht}.
\end{proof}
%%%%%%%%%%%%%%%%%%%%%%%%%%%%%%%%%%%%%%%%%%%%%%%%%%%%%%%%%%%%%%%%%%%%%%%%%%%%%%%%%%%%%%%%%
\subsection{Proof of Proposition \ref{reduction}}\label{pf}
Let us finish the proof of Proposition \ref{reduction}.
This follows from the propositions above and induction in stages.
Now assume that $w \in \wl_\phi(M, \Mp(W))$, and let
\begin{equation*}
w=w_1 \cdots w_l
\end{equation*}
be a reduced decomposition of $w$ in $\wl(\hat{M}, \Sp_{2n}(\C))$.
Then it can be seen that
\begin{align*}
&\calR_P(w, \pi_M, \psi)=\calR_P(w_1, \pi_M, \psi) \circ \cdots \circ \calR_P(w_l, \pi_M, \psi),&
&\calR_Q(w, \sigma_L)=\calR_Q(w_1, \sigma_L) \circ \cdots \circ \calR_Q(w_l, \sigma_L).&
\end{align*}
Thus, it suffices to show that the following diagram commutes for any simple reflection $w \in \wl(\hat{M}, \Sp_{2n}(\C))$:
\begin{align*}
\begin{CD}
  \omega_{V, W, \psi} \otimes \Ind_Q^{\SO(V)}(\breve{\sigma}_L) @>\calT>> \Ind_P^{\Mp(W)}(\pi_M) \\
   @V 1\otimes\mathcal{R}_Q(w, \breve{\sigma}_L)VV                      @VV\mathcal{R}_P(w, \pi_M, \psi)V \\
  \omega_{V, W, \psi} \otimes \Ind_Q^{\SO(V)}(w\breve{\sigma}_L) @>\calT>> \Ind_P^{\Mp(W)}(w\pi_M).
\end{CD}
\end{align*}
Recall the realization $\wl(\hat{M}, \Sp_{2n}(\C)) \inj \wlbc{m}$.
The commutativity follows from Lemma \ref{adjacent} and the equation $\calT = \Phi\inv\circ\calT^a\circ(1\otimes\Phi)$ when $w \in \mathfrak{S}_m$,
 and from Lemma \ref{reflection}, Proposition \ref{comm}, and the equation $\calT=\Psi\inv\circ\calT^r\circ(1\otimes\Psi)$ repeatedly when $w\in (\Z/2\Z)^m$.
Then we have completed the proof of Proposition \ref{reduction}.
%%%%%%%%%%%%%%%%%%%%%%%%%%%%%%%%%%%%%%%%%%%%%%%%%%%%%%%%%%%%%%%%%%%%%%%%%%%%%%%%%%%%%%%%%%%%%%
%%%%%%%%%%%%%%%%%%%%%%%%%%%%%%%%%%%%%%%%%%%%%%%%%%%%%%%%%%%%%%%%%%%%%%%%%%%%%%%%%%%%%%%%%%%%%%

\end{document}